\documentclass[reqno]{amsart}
\usepackage{array}   
\usepackage{amsmath,amssymb,amsthm,amsfonts }
\usepackage{bm}
\usepackage{cite}
\usepackage{color,xcolor}
\usepackage{enumerate}
\usepackage{graphicx}
\usepackage{geometry}  
\usepackage{hyperref}
\usepackage{stmaryrd}
\SetSymbolFont{stmry}{bold}{U}{stmry}{m}{n}
\usepackage{tabu}
\usepackage{tabularx}
\usepackage{tikz} 
\usepackage{url} 
\usepackage[normalem]{ulem}
\newtheorem{theorem}{Theorem}[section]
\newtheorem{lemma}{Lemma}[section]

\newtheorem{remark}{Remark}[section]
\newtheorem*{example}{Example}

\numberwithin{equation}{section}

\newcommand{\energy}[1]{\interleave #1 \interleave_{\gamma}}
\newcommand{\norme}[1]{\interleave #1\interleave}
\newcommand{\norm}[1]{\left\Vert#1\right\Vert}
\newcommand{\abs}[1]{\left\vert#1\right\vert}
\newcommand{\bra}{\langle}
\newcommand{\ket}{\rangle}
 
\newcommand{\jv}[1]{\llbracket #1 \rrbracket}

\newcommand{\E}{\boldsymbol{E}}

\newcommand{\f}{\boldsymbol{f}}
\newcommand{\g}{\boldsymbol{g}}
\newcommand{\Mfg}{\boldsymbol{M}(\boldsymbol{f,g})}
\newcommand{\ii}{\boldsymbol{ \mathrm{i} }}

\newcommand{\V}{\boldsymbol{V}}

\newcommand{\Vh}{\boldsymbol{V}_h}
\newcommand{\J}{\boldsymbol{J}}
\newcommand{\bO}{\mathcal{O}}

\newcommand{\Th}{\mathcal{T}_h}

\newcommand{\Eh}{{\E}_h}
\newcommand{\Eht}{{\E}_{h,T}}
\newcommand{\PEp}{\boldsymbol{P}^{+}_h\E } 

\newcommand{\ps}{{\boldsymbol{\Psi}}}
\newcommand{\pps}{{\boldsymbol{P}^{-}_h\boldsymbol{\Psi}}}

\newcommand{\w}{{\boldsymbol{w}}}
\newcommand{\wh}{{\boldsymbol{w}_h}}

\newcommand{\bL}{{\boldsymbol{L}}}
\newcommand{\bH}{{\boldsymbol{H}}}

\newcommand{\br}{{\boldsymbol{r}}}
\newcommand{\bv}{{\boldsymbol{v}}}
\newcommand{\bu}{{\boldsymbol{u}}}

\newcommand{\vh}{{\boldsymbol{v}_h}}

\newcommand{\bbeta}{{\boldsymbol{\eta} }}
\newcommand{\bxi}{{\boldsymbol{\xi} }}

\newcommand{\etal}{\textit{et al}.}

\newcommand{\eq}[1]{\begin{align}#1\end{align}}
\newcommand{\eqn}[1]{\begin{align*}#1\end{align*}}

\newcommand{\ls}{\lesssim}

\newcommand{\De}{\Delta}
\newcommand{\n}{\boldsymbol{\nu}}
\newcommand{\ga}{\gamma}
\newcommand{\Ga}{\Gamma}
\newcommand{\la}{\lambda}

\newcommand{\na}{\nabla}

\newcommand{\Om}{\Omega}
\newcommand{\pa}{\partial}

\newcommand{\si}{\sigma}

\newcommand{\vp}{\varphi}
\newcommand{\ka}{\kappa}

\DeclareMathOperator{\re}{{Re}}
\DeclareMathOperator{\im}{{Im}}
\DeclareMathOperator{\curl}{\boldsymbol{\mathrm{curl}}}
\DeclareMathOperator{\dive}{\boldsymbol{\mathrm{div}}}
\DeclareMathOperator{\diam}{{diam}}

\newcommand{\qaq}{\quad\mbox{and}\quad}

\newcommand{\fa}{\mathsf{f}}

\begin{document}
	
	\title[EEM   and  CIP-{EEM} for time-harmonic Maxwell Equations]{Preasymptotic error estimates of 
		EEM   and  CIP-EEM   for the  time-harmonic Maxwell {equations with large} wave number}
	\markboth{ S. Lu and H. Wu}{EEM & CIP-EEM for time-harmonic Maxwell Equations}
	
	\author[S. Lu]{Shuaishuai Lu}
	\address{Department of Mathematics, Nanjing University, Jiangsu, 210093, P.R. China. }
	\curraddr{}
	\email{ssl@smail.nju.edu.cn}
	\thanks{This work was partially supported by the NSF of China under grants 12171238, 12261160361, and 11525103.}
	
	\author[H. Wu]{Haijun Wu}
	\address{Department of Mathematics, Nanjing University, Jiangsu, 210093, P.R. China. }
	\curraddr{}
	\email{hjw@nju.edu.cn}
	\thanks{}
	
	
	\date{}
	
	\dedicatory{}
	
	
	\begin{abstract}
		Preasymptotic error estimates are derived for the linear edge element method (EEM) 
		and the linear  $\bH( \curl )$-conforming interior penalty edge element method  (CIP-EEM) for the time-harmonic
		Maxwell equations with large wave number. 
		It is shown that  under the mesh condition that $\ka^3 h^2$ is
		sufficiently small, the errors of the solutions to both methods  
		are bounded by  $\bO (\ka h + \ka^3 h^2 )$  in the energy norm and $\bO (\ka h^2 + \ka^2 h^2 )$ in 
		the $\bL^2$ norm, where $\ka$ is the wave number and $h$ is the mesh size.  Numerical  tests are provided
		to verify our theoretical results and to illustrate the potential of CIP-{EEM} in significantly reducing the
		pollution effect.
		
	\end{abstract}
	
	\maketitle
	
	\section{Introduction}\label{sc1}
	In this paper, we consider the following time-harmonic Maxwell equations for the electric field $\E$ with 
	the standard impedance boundary condition: 
	\begin{alignat}{2} 
		\curl \curl  \E  - \ka^2 \E &= \f \qquad  &&{\rm in }\  \Omega,  \label{eq:eq}  \\
		\curl \E \times  \n -\ii\ka\la\E_T &=\g \qquad  &&{\rm on }\ \Ga := \pa  \Omega,  \label{eq:bc}
	\end{alignat}
	where $\Omega \subset \mathbb{R}^{3}$ is a bounded  domain with a $C^2$ boundary. Here, $\ii=\sqrt{-1}$ 
	denotes the imaginary unit while $\n$ denotes the unit outward normal to $\Ga$ 
	and $\E_{T}:=(\n \times \E) \times \n$ is the tangential component 
	of   $\E$ on $\Ga$. Additionally, 
	$\ka>0$ is the wave number, and $\la>0$ is known as the impedance constant. 
	The right-hand side $\f$ is related to a given current density (cf. \cite{monk2003}). 
	It is assumed that {$\dive\f=0$ in $\Om$ and $\g_{T}=\g$ on $\Ga$ (i.e. $\g\cdot \n=0$).} 
	The time-harmonic Maxwell equations govern the propagation of electromagnetic waves 
	at a specified frequency, serving as fundamental tools
	for comprehending and forecasting electromagnetic field behaviors across diverse
	applications such as telecommunications, radar, optics, and electromagnetic sensing.
	Accurate numerical approximation of these equations is imperative for predicting
	wave propagation phenomena in practical scenarios.
	
	Since the seminal contributions of N\'{e}d\'{e}lec \cite{Nedelec1980,Nedelec1986}, the  
	edge element method (EEM) has garnered significant popularity and become a central tool  for addressing electromagnetic field problems. 
	Subsequently, a substantial body of literature has emerged studying EEM for solving the  Maxwell equations, we refer the readers to \cite{monk2003,Hiptmair2002femcem,jin2015theorycem} and the references therein. For time-harmonic Maxwell problems with impedance boundary conditions, 
	Monk \cite[Chapter~7]{monk2003} and Gatica  {and Meddahi} \cite{gatica2012finite} have investigated the convergence and error estimation of EEM under
	the condition that the mesh size is sufficiently small. Both of these results  are $\ka$-implicit, meaning they do not discuss the
	{influence} of $\ka$ on the error, and thus are applicable primarily to problems with  {small} wave numbers.
	
	However, in scenarios involving large wave numbers,  {a large number of studies of the finite element methods (FEMs) for Helmholtz equations \cite{Babuska1997pollution,ihlenburg98,melenk2010dtn,melenk2011,Wu2014Pre,DuWu2015,ZhuWu2013II,ZhuWuHDGHelmholtz} indicate that the EEM of fixed order may also suffer from the well-known ``pollution effect", i.e., compared to the best approximation from the discrete space, the approximation ability of the solution to the EEM (with fixed order) gets worse and worse as the wave number $\ka$ increases. 
		Clearly, estimating the pollution error is significant both in theory and practice, and it has always been interesting to develop numerical methods with less pollution error. Unfortunately, rigorous pollution-error analyses for the EEM for the time-harmonic Maxwell equations are still unavailable in the literature. We recall that the term ``asymptotic error estimate" refers to the error estimate without pollution error and the term ``preasymptotic error estimate" refers to the estimate with non-negligible pollution effect.}

	The latest asymptotic error analysis for the time-harmonic Maxwell equations was conducted by Melenk and Sauter in
	\cite{Melenk_2020,melenk2023wavenumber}, where the {first type N\'{e}d\'{e}lec} $hp$-EEM for solving the 
	time-harmonic Maxwell equations with transparent boundary condition and impedance  boundary condition
	has been studied. Using the so-called  ``regularity decomposition" technique, they show 
	that the $hp$-EEM is pollution-free under the following
	\emph{scale resolution conditions}:
	\begin{equation}
		\frac{  \ka  h}{p}\leq c_{1}\qquad\mbox{ and }\qquad p\geq
		c_{2}\ln \ka , \label{eq:scale-resolution} 
	\end{equation}
	where $h$ is the mesh size, $c_{2}$ is an arbitrary positive  number and $c_{1}>0$ is sufficiently small. The resolution 
	condition $\ka h/p\leq c_{1}$ is natural for resolving the oscillatory
	behavior of the solution, and the side constraint $p\geq c_{2}\ln 
	\ka $ suppresses the
	pollution effect. 
	The above result highlights  the advantages of high-order EEM, representing a breakthrough in the
	analysis of the time-harmonic Maxwell  equations  in the large wave number regime. {But for fixed order, say second order EEM of the first type, the above two works require that $k^8h$ is sufficiently small  (see \cite[Remark 4.19]{Melenk_2020}) for the case of transparent boundary condition and $k^2h$ is sufficiently small for the case of impedance boundary condition, respectively, which are too strict in $h$.
		For the special case $\g = \bm{0} $ in problem \eqref{eq:eq}--\eqref{eq:bc}, Nicaise and
		{Tomezyk \cite{nicaise2020convergence,NicaiseTomezyk2019}}, using $h$- and $hp$-FEM  with Lagrange elements, performed
		asymptotic error analyses for a regularized Maxwell system(cf. \cite[\S~4.5d]{costabel2010corner}).

		Recently, Lu~ \etal \cite{pplu2019} proposed  {an}  $\bH( \curl  )$-conforming  interior penalty  edge element method 
		(CIP-EEM), which uses the same approximation space as EEM but modifies the sesquilinear form by adding a least squares 
		term penalizing the jump of the tangential {components of the} $ \curl $ of the discrete solution at mesh interfaces.
		Using the so-called  ``stability-error iterative improvement" technique \cite{fw2011},  they  proved that the CIP-EEM with pure imaginary penalty parameter $\ga=-\ii\ga_{\ii}$ is absolute stable (that is, stable for any $\ka$ and $h$) and obtained the 
			following error estimate  under the conditions $\ga_{\ii} \simeq 1$ and $\ka^3 h^2 =\bO(1) $: 
			\begin{equation}
				\norme{\E-\Eh} =\bO  (\ka h+ \ka^3 h^2),
				\label{eq:PreenergyResult0}
			\end{equation}
			where $\norme{\cdot}=\big(\norm{\curl \cdot}^2+\ka^2\|\cdot\|^2 + \ka\la\|\cdot_T\|^2 _{ \Ga }\big)^{\frac{1}{2}}$ is the energy norm. This is a closely related result. However, this result does not encompass the  classical  EEM   since $\ga$ cannot be zero there.   For wave-number-explicit error analyses of various discontinuous Galerkin (DG) methods, we refer to \cite{feng2014absolutely,feng_lu_xu_2016,hiptmair2012}.

		The purpose of this paper is twofold. First, we focus on the preasymptotic error analysis of the EEM and provide the first such results in the literature for the EEM using the linear N\'{e}d\'{e}lec element of the second type. To be precise, we have proved the following error estimate for the solution $\Eh$ to the linear EEM,  under the mesh condition that $\ka^3 h^2 $ is sufficiently small: 
		\begin{equation}\label{eq:intro_energy_error}
			\norme{\E-\Eh} =  \bO (\ka  h + \ka^3 h^2 ).
		\end{equation} 
		Clearly,  the first term $\bO(\ka h)$ on the right-hand side of \eqref{eq:intro_energy_error} is of the same order as the error of the best approximation, which can be reduced by putting enough points per wavelength. However, the second term $\bO(\ka^3h^2)$ will get out of control as the wave number $\ka$ increases, no matter how many points are put per wavelength, as long as the number of them is fixed.  Therefore, the second term is the pollution error which dominates the error bound when  $\ka^2 h$ is large. Moreover, the following $\bL^2$ error estimate has also been proved under the same mesh condition:
		\eq{\label{eq:intro_L2_error}\ka\|\E-\Eh\|  =\bO \big( (\ka h)^2 + \ka^3 h^2\big). }
		We would like to mention that, in order to prove the error estimates \eqref{eq:intro_energy_error}--\eqref{eq:intro_L2_error}, thanks to the recent work \cite{chen24}, we derive a wave-number-explicit $\bH^2$ regularity estimate (see \eqref{eq:H2stab}) for the continuous problem on the domain with only  $C^2$-smooth boundary. A similar estimate was proved in \cite{lu2018regularity} for a domain with $C^3$ boundary.     
		Secondly, we  extend the result in \cite{pplu2019} for the CIP-EEM to a more general case, i.e., the linear CIP-EEM with complex penalty parameters $\ga=\ga_r-\ii\ga_i$. To be precise,  we prove that \eqref{eq:PreenergyResult0} still holds,
		if  $\re\gamma \geq -\alpha_0 $, $\im\gamma\leq 0$,  $|\gamma|\le C$, and $\ka^3 h^2 \leq C_0$, for some positive constants $C_0$ and $\alpha_0$ independent of $\ka$ and $h$. Such an extension  is meaningful since the penalty parameters that can greatly reduce the pollution error are usually close to negative real numbers (see Section~\ref{sc6}), which is not covered by the theory given in \cite{pplu2019}. By the way, the divergence-free constraint on the source term $\f$ in \cite{pplu2019} is dropped in this paper. Compared with the linear EEM, the linear CIP-EEM provides a candidate  with enhanced stability and low pollution effects. We would like to mention that due to their easy of implementation, low-order methods are still attractive in many applications, especially for those low-regularity problems where a higher-order method cannot achieve its full convergence rate.  
		
		The key idea in our preasymptotic analysis for the EEM and CIP-EEM is to first prove stability estimates for the discrete solution $\Eh$, in which a delicate decomposition of $\Eh$ (see \eqref{eq:Ehdecomp}) and a so-called ``modified duality argument" technique play  important roles, and then use the obtained stability estimates to derive the desired error estimates. The ``modified duality argument" technique was first used to derive preasymptotic error estimates for the FEM for Helmholtz equations with large wave numbers \cite{ZhuWu2013II,DuWu2015}. Here we extend it to derive stability estimates for the EEM. We would like to mention that  our proofs are highly non-trivial.  The ``stability-error iterative improvement" technique used in \cite{pplu2019} does not work for the EEM or CIP-EEM with general parameters.  We had also tried to derive preasymptotic error estimates for the EEM by mimicking the usual process (see, e.g., \cite{monk2003}), including using the modified duality argument to derive the error estimates directly, but failed. 
		The preasymptotic error analysis of higher-order EEM and CIP-EEM necessitates additional technical tools  and will be investigated in another work.

		The remainder of this paper is organized as follows. In Section \ref{sc2}, we formulate the   EEM   and the CIP-EEM.
		Section \ref{sc3} introduces some preliminary results about the stability estimates 
		and some error estimates of some $\bH( \curl  )$-elliptic projections. In  Section \ref{sc4}, we first establish
		discrete stability estimates  for the EEM solution.
		Then we prove the preasymptotic error estimates of the EEM for the time-harmonic Maxwell problem.
		In Section \ref{sc5},  we extend the results of  preasymptotic error estimates to the linear CIP-EEM with general penalty parameters. 
		Finally, in Section \ref{sc6},  we present some numerical examples to verify our theoretical findings and the great potential of the CIP-EEM to reduce the pollution errors.

		Throughout this paper, we use notations $A \ls B$ and $A \gtrsim B$ for
		the inequalities $A \leq C B$ and $A \geq C B$,
		where $C$ is a positive number independent of the mesh size $h$ and the wave number $\ka$,
		but the value of which can vary in different occurrences. 
		$A \simeq B$ is a shorthand notation for the statement $A \leq C B$ and $B \leq C A$. 
		{We also denote by $C_{\ka h}$ a generic constant which may depends on $\ka h$ 
			but satisfies $C_{\ka h} \ls 1$ when $\ka h\ls 1$}.
		
		For simplicity, we suppose $\la \simeq 1$, $\ka{\gtrsim1}$. We also assume that the domain $\Om$ 
		is strictly star-shaped with respect to a point ${\bm x}_0\in\Om$, that is,
		$({\bm x}-{\bm x}_0)\cdot \n\gtrsim 1 \quad \forall\bm x\in\Ga$.

		\section{Formulations of   EEM   and  CIP-EEM }\label{sc2}
		
		In order to formulate the two methods, we first introduce some notations. For a domain $D\subset \mathbb{R}^3$ with Lipschitz 
		boundary $\pa D$,  {we shall use the standard Sobolev space $ H^s(D)$, its norm $\|\cdot\|_{s,D}$, seminorm 
			$|\cdot|_{s,D}$, and inner product. We refer to \cite{adams2003sobolev,brenner2008mathematical,monk2003} for their definitions.}   
		If the functions are vector-valued we {shall} indicate these function spaces 
		by boldface symbols{, e.g.,} $\bH^s(D)$. 
		In particular, $(\cdot, \cdot)_{D}$ and $\bra\cdot, \cdot \ket_{\Sigma}$ for
		$\Sigma \subset \pa D$
		denote the $\bL^2$-inner product on complex-valued $\bL^2(D)$ and $\bL^2(\Sigma)$ spaces, respectively.
		{For simplicity, we shall use} the shorthands:
		\begin{align*}
			(\cdot, \cdot):=(\cdot, \cdot)_{\Omega}, \; \bra\cdot, \cdot \ket:=\bra\cdot, \cdot \ket_{\Ga},\; |\bv|_s:= |\bv|_{s,\Om}, \;
			\|\bv\|_s:= \|\bv\|_{s,\Om}, \;	\|\bv\| := \|\bv\|_{0,\Om}, \;\text{and}\;	\|\bv\|_{\Ga}:= \|\bv\|_{0,\Ga}.
		\end{align*}
		We introduce the space
		\eqn{
			& \bH(\curl; \Omega) :=\{\bv\in  \bL^2{(\Omega)} : \curl\bv\in  \bL^2{(\Omega)} \},\\
			& \bH^1(\curl; \Omega) :=\{\bv\in  \bH^1{(\Omega)} : \curl\bv\in  \bH^1{(\Omega)} \}
			\quad\text{with norms}\\
			& \norm{\bm v}_{\bH(\curl)}:=\big(\norm{\curl\bm v}^2+\norm{\bm v}^2\big)^\frac12, \\
			& \norm{\bm v}_{\bH^1(\curl)}:=\big(\norm{\curl\bm v}_1^2+\norm{\bm v}_1^2\big)^\frac12   
			\quad\text{and seminorm} \\
			& |\bm v|_{\bH^1(\curl)} := \big(|\curl\bm v|_1^2+|\bm v|_1^2\big)^\frac12.  
		}
		Let
		\eqn{
			H_0^1(\Om) &:= \big\{v \in H^1(\Om) :  v|_{\Ga} =0\big\},\\
			\bH_{0}(\curl; \Omega )& :=\big\{\bv \in \bH(\curl; \Omega ) : \bv \times \n|_\Ga=0 \big\}, }
		where $v|_{\Ga}$ is understood as the trace from $H^1(\Om)$ to $H^\frac12(\Ga)$ and $\bv \times \n|_\Ga$  
		is the tangential trace from $\bH(\curl; \Omega )$ to $\bH_{\dive}^{-1/2}(\Ga)$. Similarly, we will understand 
		$\bv_T$ as the tangential trace from $\bH(\curl; \Omega )$ to $\bH_{\curl}^{-1/2}(\Ga)$  (see, e.g., \cite{monk2003,Melenk_2020}).

		Next we define the  ``energy"  space
		\eqn{
			\V&:= \big\{ \bv\in \bH(\curl; \Om) :  {\bv_T \in \bL^2(\Ga) \big\}}  \quad\text{with norm}\\
			{\norme{\bv}}&:=\big(\norm{\curl\bm v}^2+\ka^2\|\bv\|^2 +
			\ka\la\|\bv_T\|^2 _{ \Ga }\big)^{\frac{1}{2}},}
		and introduce the following sesquilinear form on $\V\times \V$: 
		\begin{equation}
			a\left({\bm u} , {\bm v} \right):=\left(\curl {\bm u} , \curl {\bm v}\right) 
			-\ka^{2}\left({\bm u}, {\bm v} \right)-\ii \ka \lambda\left\langle{{\bm u}_T,{\bm v}_T}\right\rangle.
			\label{eq:a}
		\end{equation}
		Then the variational formulation of \eqref{eq:eq}--\eqref{eq:bc} reads as:
		Find $\boldsymbol{E}  \in \boldsymbol{V} $ such that
		\begin{equation}
			a \left(\E  , {\bm v} \right)=\left(\boldsymbol{f}, 
			{\bm v} \right)+\left\langle \g,{{\bm v}_T}\right\rangle 
			\quad \forall {\bm v} \in {\boldsymbol{V}}.
			\label{eq:dvp}
		\end{equation}
		
		To {discretize \eqref{eq:dvp}}, we follow \cite{melenk2023wavenumber} to introduce a regular and quasi-uniform
		triangulation $\mathcal{T}_{h}$ of $\Om$ which satisfies:
		\begin{enumerate}
			\item {The (closed)  elements} $K\in \Th$ cover $\Om$, i.e.,  
			{$\bar{\Om}=\cup_{K\in\Th}K$.}  
			\item Associated with each element $K$ is the element map,  a $C^1$-diffeomorphism 
			$F_K:\hat{K}\rightarrow K$. The set $\hat{K}$ is the reference tetrahedron. Denoting $h_K:=\diam(K)$, 
			there holds, with some shape-regularity constant $\si$,
			\eqn{h_K^{-1}\norm{F_K'}_{L^\infty(\hat K)}+h_K\norm{(F_K')^{-1}}_{L^\infty(\hat K)}\le \si,} 
			where $F_K'$ is the Jacobian matrix of $F_K$.
			\item The intersection of two elements is only empty, a vertex, an edge, a face, or they coincide
			(here, vertices, edges, and faces are the images of the corresponding
			entities on the reference element $\hat{K}$). The parameterization of common edges
			or faces is compatible. That is, if two elements $K$, $K'$ share an edge (i.e.,  
			$F_K(e) = F_{K'}(e')$ for edges $ e , e'$ of $\hat{K}$ ) or a face 
			(i.e., $F_K ( {\fa} ) = F_{K'} ( \fa' )$ for faces $\fa , \fa'$ of $\hat{K}$), 
			then $F^{-1}_K \circ  F_{K'}: \fa' \rightarrow \fa $ is an affine isomorphism.
		\end{enumerate}
		The set of all faces of $\Th$ is denoted by $\mathcal{F}_h$ , and the set of all interior faces 
		by $\mathcal{F}_h^I$. For any $\fa\in\mathcal{F}_h$, {let $h_\fa:=\diam(\fa)$.}  Denote by $h:=\max_{K\in\Th} h_K$. 
		
		Let $\V_{h}$ 
		be the first order N\'{e}d\'{e}lec edge element space of the second type (see \cite[(3.76)]{monk2003}):
		\begin{equation}
			\V_{h}:=\Big\{\bv_{h} \in \V :\; ( F_K'  )^T  ( \left.\bv_{h}\right|_{K} )\circ F_K
			\in\big(\mathcal{P}_{1}(\hat{K})\big)^{3} \quad \forall K \in \mathcal{T}_{h}\Big\},	
			\label{eq:FEMspace}
		\end{equation}
		where $\mathcal{P}_{1}(\hat{K})$ denotes the set of linear polynomials on $\hat{K}$.
		The EEM for the Maxwell problem \eqref{eq:eq}--\eqref{eq:bc} {reads as:} 
		Find $\boldsymbol{E}_{h} \in \boldsymbol{V}_{h}$ such that
		\begin{equation}
			a\left(\Eh , {\bm v}_h\right)=\left(\boldsymbol{f}, 
			{\bm v}_h\right)+\left\langle \g,{\bm v}_{h,T}\right\rangle 
			\quad \forall {\bm v}_h \in \boldsymbol{V}_{h}.
			\label{eq:EEM}
		\end{equation}
		To formulate the CIP-EEM, we define the jump of the
		tangential {components of $\curl \bu$} on an interior face 
		$\fa =\pa K^{-}\cap\pa K^{+}$ :
		\begin{equation}
			{\jv{\curl \bu }} :=  \curl\bu|_{K^{-}} \times\n_{K^{-}} +\curl\bu|_{K^{+}} \times\n_{K^{+}} ,
			\label{def:jump}
		\end{equation}
		where $\n_{K^{-}}$ and $\n_{K^{+}}$ denote the unit outward normal 
		to  $\pa K^{-}$ and $\pa K^{+}$, respectively. We introduce the following 
		sesquilinear form consisting of penalty terms on interior faces:
		\begin{equation}
			\J({\bm u},{\bm v}):= \sum_{\fa  \in\mathcal{F}_h^I}\ga_{\fa }	h_\fa 
			\bra \jv{\curl{\bm u}},\jv{\curl{\bm v}} \ket_{\fa },
			\label{eq:penalty}
		\end{equation}
		where $\ga_{\fa}$ are penalty parameters that are allowed to be complex-valued. Then the CIP-EEM is
		constructed by simply adding $\J(\Eh,{\bm v}_h)$ to the left-hand side of EEM \eqref{eq:EEM}, which is to  
		find $\boldsymbol{E}_{h} \in \boldsymbol{V}_{h}$ such that
		\begin{equation}
			a_\ga\left(\Eh , {\bm v}_h\right)=\left(\boldsymbol{f}, 
			{\bm v}_h\right)+\left\langle \g,{\bm v}_{h,T}\right\rangle
			\quad \forall {\bm v}_h \in \boldsymbol{V}_{h}	,
			\label{eq:disvp}
		\end{equation}
		where
		\begin{equation}
			a_\ga\left({\bm u}, {\bm v}\right):=a\left({\bm u}, {\bm v}\right)+\J\left({\bm u}, {\bm v}\right).	\label{eq:ah}
		\end{equation}
		
		\begin{remark}\label{rem:ppvalue}
			
			{{\rm (a)}  If $\gamma_\fa\equiv 0$, then the CIP-EEM becomes the classical   EEM  \eqref{eq:EEM}.}
			
			{\rm (b)} {The CIP-EEM is a generalization to Maxwell equations of the continuous 
				interior penalty finite element method (CIP-FEM) for elliptic and parabolic 
				problems \cite{douglas2008interior}, in particular, the
				convection-dominated problems  \cite{burman2005} and Helmholtz equations
				\cite{ZhuWu2013II,Wu2014Pre,DuWu2015}.}
			
			{\rm (c)} {The CIP-EEM was first introduced and studied in \cite{pplu2019} but with 
				pure imaginary penalty parameters. This paper concerns the CIP-EEM with general
				complex penalty parameters, especially real penalty parameters, which can help to
				reduce the pollution errors.}

			{\rm (d)} Compared to the discontinuous Galerkin
			methods \cite{feng2014absolutely,houston2005interior} and  hybridizable 
			discontinuous Galerkin method \cite{lu2017HDGMaxwell}, the  CIP-EEM  involves
			fewer degrees of freedom, and thus requires less  computational cost.
			
			{\rm (e)} {It is clear that $\J(\E,{\bm v}_h)=0$ for the exact solution $\E$, therefore, 
				the CIP-EEM is consistent with the Maxwell problem \eqref{eq:dvp}, and hence there 
				holds the following Galerkin orthogonality:}
			\begin{equation}
				a_\ga\left(\boldsymbol{E}-\Eh , {\bm v}_h\right)=0 \quad
				\forall {\bm v}_h \in \boldsymbol{V}_{h}.	
				\label{eq:orthogonality}
			\end{equation}

		\end{remark}

		\section{Preliminary}\label{sc3}
		In this section, we list some preliminary results which will be used in the 
		later sections.
		
		First of all, we list two lemmas below to give  stability and regularity  
		estimates for the Maxwell problem \eqref{eq:eq}--\eqref{eq:bc} that are explicit with respect to the wave number.

		\begin{lemma}\label{thm:stability}
			Assume that $\Omega \subset \mathbb{R}^{3}$ is a bounded $C^{2}$-domain and strictly star-shaped
			with respect to  a point ${\bm x}_{0} \in \Om$ and that $\dive\f=0$ in $\Om$ and $\g\cdot\n=0$ on $\Ga$.	
			Let $\E$ be the solution to the problem \eqref{eq:eq}--\eqref{eq:bc}. Then we have
			\eq{
				\|\curl  \E \| + \ka \|\E \|  +\ka \|\E_T \|_{\Ga} 
				&\lesssim   \|\f\| +\|\g\|_{\Ga},    \label{eq:stability0}\\
				\|\curl  \E \|_{1} +\ka\norm{\E}_1&\lesssim \ka\left(\|\f\| +\|\g\|_{\Ga}\right)+\|\g\|_{\frac{1}{2}, \Ga}, 
				\label{eq:stability1}\\
				\|\E\|_{2} &\lesssim \ka \Mfg, \label{eq:H2stab}
			}	 
			where
			\eqn{
				\Mfg:=\|\f\| +\ka^{-2}\|\f\cdot\n\|_{\frac{1}{2},\Ga} +\|\g\|_{0, \Ga}+\ka^{-1}\|\g\|_{ \frac{1}{2},  \Ga }+\ka^{-2}\|\dive_\Ga\g\|_{ \frac{1}{2},  \Ga },}
			and $\dive_\Ga$ is the the surface divergence operator (see, e.g., \cite{monk2003,melenk2023wavenumber}).	
		\end{lemma}
		\begin{proof} For the proofs of \eqref{eq:stability0} and  \eqref{eq:stability1} we refer to \cite[Theorem~3.1 and Remark 4.6]{Hiptmair2011STABILITYRF}.  
			
			In order to prove \eqref{eq:H2stab}, we rewrite the Maxwell equations \eqref{eq:eq}--\eqref{eq:bc} as
			\begin{alignat}{2} 
				\curl \curl\E- \E  &= \f+(\ka^2-1)\E   \qquad &&{\rm in }\  \Omega,    \\
				\curl\E \times  \n -\ii\la\E_T &=\g +\ii(\ka-1)\la\E_T  \qquad  &&{\rm on }\ \Ga.
			\end{alignat} 
			Recall that $\dive_\Ga(\bm v\times\n)=\n\cdot\curl\bm v$ (see, e.g., \cite[(3.52)]{monk2003}). We have from \eqref{eq:bc} and \eqref{eq:eq} that
			\eqn{
				\dive_\Ga\big(\ii\ka\la\E_T+\g\big)&=\dive_\Ga(\curl \E \times  \n)=(\curl\curl \E)\cdot\n=(\ka^2 \E+\f)\cdot\n,} which implies that
			\eqn{
				\dive_\Ga(\ii\la\E_T)=\ka \E\cdot\n+\ka^{-1}\big(\f\cdot\n-\dive_\Ga\g\big).} Therefore, from \cite[Theorem~1]{chen24}, we have 
			\eqn{
				\|\E\|_2&\ls \big\|\f+(\ka^2-1)\E\big\|+\big\|\g +\ii(\ka-1)\la\E_T\big\|_{\frac12,\Ga}\\
				&\quad+\big\|(\f+(\ka^2-1)\E)\cdot\n-\dive_\Ga\g -\ii(\ka-1)\dive_\Ga(\la \E_T)\big\|_{\frac12,\Ga}\\
				&=\big\|\f+(\ka^2-1)\E\big\|+\big\|\g +\ii(\ka-1)\la\E_T\big\|_{\frac12,\Ga}\\
				&\quad+\big\|\ka^{-1}(\f\cdot\n-\dive_\Ga\g)+(\ka-1)\E\cdot\n\big\|_{\frac12,\Ga}\\
				&\ls \|\f\|+\ka^2\|\E\|+\|\g\|_{\frac{1}{2},\Ga}+\ka\|\E\|_1+\ka^{-1}\|\f\cdot\n\|_{\frac{1}{2},\Ga}+\ka^{-1}\|\dive_\Ga\g\|_{ \frac{1}{2},  \Ga },
			}
			where we have used the trace inequality $\|\bv\|_{\frac12,\Ga}\ls \|\bv\|_1$ to derive the last inequality. 
			Then \eqref{eq:H2stab} follows by combining \eqref{eq:stability0}--\eqref{eq:stability1} and the above estimate. This completes the proof of the lemma.
		\end{proof}
		
		For the estimate of $\|\curl  \E \| + \ka \|\E \|$ in \eqref{eq:stability0}, we refer to 
		\cite[Proposition~3.6]{melenk2023wavenumber} and \cite[Theorem~3.1]{Hiptmair2011STABILITYRF}. The estimate of $\ka \|\E_T \|_{\Ga}$ in 
		\eqref{eq:stability0} and  \eqref{eq:stability1} can be obtained by using \cite[Theorem~3.1 and Remark 4.6]{Hiptmair2011STABILITYRF}
		and mimicking the proof of \cite[Proposition~3.6]{melenk2023wavenumber}.  
		\begin{remark}\label{rem3.1} {\rm (a)} The regularity estimate \eqref{eq:H2stab} was proved in 
			\cite[Theorem 3.4]{lu2018regularity} when the domain $\Om$ is $C^{3}$ and in \cite[Lemma~5.1]{melenk2023wavenumber} when $\Om$ is sufficiently smooth.
			
			{\rm (b)} \cite[Theorem~1]{chen24} gave a regularity estimate by assuming  only $C^{2}$-domain, 
			however, explicit dependence on the wave number $\ka$ was not considered there.
			
			{\rm (c)} If $\dive\f\neq 0$, we can obtain the estimates of $\E$ as follows.
			Inspired by the proof of \cite[Proposition~3.6]{melenk2023wavenumber}, let $\vp\in H_0^1(\Om)$ satisfy
			\eqn{-\De\vp=\dive\f.}
			Let $\bm u=\E-\ka^{-2}\na\vp$. Clearly, $\bm u$ satisfies
			\begin{alignat}{2} 
				\curl \curl \bm u- \ka^2 \bm u &= \f+\na\vp   \qquad &&{\rm in }\  \Omega,    \\
				\curl \bm u \times  \n -\ii\ka\la\bm u_T &=\g  \qquad &&{\rm on }\  \Ga.
			\end{alignat} 
			We have $\dive(\f+\na\vp)=0$ and $\dive\bm u=0$. Therefore, we can apply Lemma~\ref{thm:stability} to 
			obtain estimates for $\bm u$ and then apply the triangle inequality to obtain estimates for $\E$. 
			In particular, the stability estimate \eqref{eq:stability0} still holds, the regularity estimate \eqref{eq:stability1} 
			holds if an additional term $k^{-1}\|\dive\f\|$ is added to its right-hand side, and the $\bH^2$ regularity 
			estimate \eqref{eq:H2stab} holds if the domain $\Om$ is $C^3$ and an additional term $\ka^{-3}\|\dive\f\|_1$ is added to $\Mfg$.  
		\end{remark}

		Next we introduce two  $\bH(\curl)$-elliptic projections and recall their error estimates, 
		which will be used in our preasymptotic error analysis for EEM. Let
		\eq{\label{eq:hat a}
			{\hat a}({\bm u}, {\bm v})= a({\bm u}, {\bm v} )+2\ka^2({\bm u}, {\bm v})=\left(\curl {\bm u} ,
			\curl {\bm v}\right) +\ka^{2}\left({\bm u}, {\bm v} \right)-\ii \ka \lambda\left\langle{\bm u}_T,{\bm v}_T\right\rangle.}  
		For any ${\bm u}\in\V$, define its   $\bH(\curl)$-elliptic  projections ${\bm P}_h^\pm{\bm u}\in\V_h$ as
		\begin{align}
			\hat{a}\big({\bm u}-{\bm P}_h^+{\bm u},{\bm v}_h\big)=0, \quad
			\hat{a}\big({\bm v}_h,{\bm u}-{\bm P}_h^-{\bm u}\big)=0    \quad \forall {\bm v}_h\in\Vh. \label{eq:proj} 
		\end{align}
		
		We have the following error estimates for the projections, whose proofs can be found in \cite{pplu2019} and are also given in the appendix for the reader's convenience.
		\begin{lemma}\label{lem:Pherror}
			Suppose that $\bm u\in\bH^2(\Om)$. Then
			\begin{subequations}
				\begin{align}
					\norme{ {\bm u}-\boldsymbol{P}^{\pm}_h{\bm u}}&\ls \inf_{{\bm v}_h\in\V_h}\norme{{\bm u}-{\bm v}_h}
					\ls h |{\curl \bm u}|_{1}+C_{\ka h}(\ka h)^{\frac{1}{2}}h |{\bm u}|_{2},
					\label{eq:Pherror_energy}\\
					\|{\bm u}-\boldsymbol{P}^{\pm}_h{\bm u}\|    &\ls C_{\ka h} h^{2 }|{\bm u}|_{2}, \label{eq:Pherror_l2} \\
					\|({\bm u}-\boldsymbol{P}^{\pm}_h{\bm u})_T\|_{\Ga} & \ls C_{ \ka h}h^{\frac{3}{2}}|{\bm u}|_2\label{eq:boundPherror}.
				\end{align}
			\end{subequations} 		
		\end{lemma}
		
		Let $\V_{h}^{0}:=\V_{h} \cap \bH_{0}(\curl; \Omega ) $. The following lemma states that any discrete function in $\V_h$
		has an  ``approximation" in $\V_h^0$, whose error can be bounded by its tangential components on $\Ga$. 
		The proof is similar to that of \cite[Proposition~4.5]{houston2005interior}, but simpler.
		\begin{lemma}\label{lem:Vh0}
			For any $\bm v_h\in \V_h$, there exists $\bm v_h^0\in \V_h^0$, such that
			\eq{\label{eq:Vh0}\|\bm v_h-\bm v_h^0\| \ls h^{\frac{1}{2}}\|\bm v_{h,T}\| _{ \Ga }\qaq \|\curl(\bm v_h-\bm v_h^0)\|
				\ls h^{-\frac{1}{2}}\|\bm v_{h,T}\| _{ \Ga }.}
		\end{lemma}
		\begin{proof} We only give a sketch of the proof for the reader's convenience. In fact, 
			denoting by $M_{e,j}(\bm v), j=1,2$ the degrees of freedom (moments) of a function 
			$\bm v$ on an edge $e$ of an element, $\bm v_h^0$ can be defined simply by changing those 
			degrees of freedom of $\bm v_h$ on the boundary $\Ga$ to 0:
			\eqn{M_{e,j}(\bm v_h^0)=
				\begin{cases}
					0  &\text{if } e \subset\Ga,\\
					M_{e,j}(\bm v_h) &\text{otherwise, }
				\end{cases}\quad j=1,2.}
			Then by some simple calculations, we obtain
			\eqn{\|\bm v_h-\bm v_h^0\|^2 &=\sum_{K\in\Th,K\cap\Ga\neq\emptyset}\norm{\bm v_h-\bm v_h^0}_K^2
				\ls \sum_{K\in\Th,K\cap\Ga\neq\emptyset} h_K    \sum_{e\subset\pa K\cap\Ga, j=1,2}|M_{e,j}(\bm v_h)|^2
				\ls h \|\bm v_{h,T}\| _{ \Ga }^2,}
			which implies the first inequality in \eqref{eq:Vh0}, which together with the inverse inequality implies
			the second one. This completes the proof of the lemma. 
		\end{proof}

		Let $U_{h}:=\{u \in H^{1}(\Omega):(u|_{K})\circ F_K \in \mathcal{P}_{2}(\hat{K}) \quad \forall K \in \mathcal{T}_{h} \}$ 
		and $U_{h}^{0}:=U_{h} \cap H_{0}^{1}(\Omega)$,
		where $\mathcal{P}_{2}(\hat{K})$ is the set of quadratic polynomials on $\hat{K}$.
		Obviously $\nabla U_{h}$ 
		and $\nabla U_{h}^{0}$ are subspaces of $\V_{h}$ and
		$\V_{h}^{0}$, respectively. The following lemma gives both the Helmholtz decomposition and the discrete
		Helmholtz decomposition for each $\bv_{h} \in \V_{h}^0$.
		
		\begin{lemma}[Helmholtz Decomposition]\label{lem:HD} \hfill
			
			For any $\vh \in \V_{h}^{0}$, there exist $r \in H_{0}^{1}(\Omega), r_{h} \in U_{h}^{0}, \w \in \bH_{0}(\curl , \Omega)$,
			and $\wh \in$ $\V_{h}^{0}$, such that
			\begin{align}
				&\vh=\nabla r+\w=\nabla r_{h}+\wh , \label{HD1}   \\
				&\dive \w =0 \quad \text { in } \Omega, \quad\left(\wh, \nabla \psi_{h}\right)=0 \quad \forall \psi_{h} \in U_{h}^{0}, \label{HD1div} \\
				&\left\|\w-\w_{h}\right\|  \lesssim h\left\|\curl  \bv_{h}\right\|.\label{HD1w}
			\end{align}
		\end{lemma}
		\begin{proof}
			This Lemma may be proved by using  \cite[Theorem~3.45, Lemma~7.6]{monk2003}.
			We omit the details.
		\end{proof}

		\section{{Preasymptotic error estimates for EEM} }\label{sc4}
		
		In this section, we  first establish
		{some stability estimates  for the EEM \eqref{eq:EEM} by using the so-called  ``modified duality argument"
			developed for the preasymptotic analysis of FEM for Helmholtz equation \cite{ZhuWu2013II}.} 
		Then we use the resulting stability estimates  to derive preasymptotic error estimates for the EEM.
		
		\subsection{Stability estimates of EEM}

		The following theorem gives the stability of the EEM solution. 
		\begin{theorem}\label{thm:EEMStability}
			Let $\E_h$	 be the solution to \eqref{eq:EEM},
			$\f\in \bL^2(\Om)$, $\g\in\bL^2(\Ga)$. There exists a constant $C_0>0$ independent 
			of $\ka$ and $h$ such that if  {$\ka^3 h^2\le C_0$},  then
			\begin{equation}
				\|\curl\E_h\| +\ka \|\E_h\|+\ka \|{\Eht}\|_{\Ga} \ls \|\f\| +  \|\g\|_{\Ga}.
				\label{eq:EEMStability}
			\end{equation} 
			And as a consequence, the EEM \eqref{eq:EEM} is well-posed.
		\end{theorem}
		
		\begin{proof} The proof is divided into the following steps.
			
			{\it Step 1. Estimating $\norme{\E_h}$ and $\|\Eht\|_{\Ga}$ by $\|\E_h\|$.}
			Taking $\bv_h = \E_h $ in \eqref{eq:EEM} and taking the imaginary and real parts on both sides, we may get
			\begin{align*}
				-\im a (\Eh ,\Eh )&=\ka \la \|\Eht\|_{\Ga}^2  \leq \|\f\| \|\E_h\|+\|\g\|_{\Ga}\|\Eht\|_{\Ga}\\  
				& \leq \|\f\| \|\E_h\| +\frac{1}{2\ka\la}\|\g\|_{\Ga}^2+ \frac{\ka\la}{2} \|\Eht\|_{\Ga}^2,\\ 
				(\re-\im) a (\Eh ,\Eh )  & = \norme{\E_h}^2 -2\ka^2 \|\E_h\|^2 =(\re-\im)\big((\boldsymbol{f}, 
				\Eh )+\langle \g,{\E}_{h,T}\rangle \big) \\  
				& \leq \sqrt{2}\Big(\|\f\| \|\E_h\| +\frac{1}{2\ka\la}\|\g\|_{\Ga}^2+ \frac{\ka\la}{2} \|\Eht\|_{\Ga}^2\Big),
			\end{align*}	
			which imply that
			\begin{align*}
				\ka \|\Eht\|_{\Ga}^2 &\ls \|\f\|\|\E_h\|+\ka^{-1}\|\g\|_{\Ga}^2, \\
				\norme{\E_h}^2  &\ls \ka^2 \|\E_h\|^2+ \|\f\| \|\E_h\| +\ka^{-1} \|\g\|_{\Ga}^2, 
			\end{align*}
			and hence
			\begin{align}
				\|\Eht\|_{\Ga} &\ls \|\E_h\|+\ka^{-1}\big(\|\f\|+\|\g\|_{\Ga}\big) , \label{eq:stabbd}\\
				\|\curl\Eh\|  &\ls \ka \|\E_h\|+ \ka^{-\frac12}\big(\|\f\|+\|\g\|_{\Ga}\big). \label{eq:stabEner}
			\end{align}
			
			{\it Step 2. Decomposing $\|\E_h\|$.} According to  Lemma~\ref{lem:Vh0},
			there exists $\Eh^0\in \bH_{0}(\curl; \Om)\cap\V_h$ such that 
			\begin{equation}
				\|\Eh-\Eh^0\| +h\|\curl(\Eh-\Eh^0)\| \ls
				h^{\frac{1}{2}}\|\Eht\| _{ \Ga }.
				\label{eq:ehcerr}
			\end{equation} 
			By Lemma~\ref{lem:HD},    we have the following discrete Helmholtz decomposition
			for $\Eh^0$:
			\begin{equation}
				\Eh^0=\w_{h}^0+\nabla r_h^0,
				\label{eq:EhcHD}
			\end{equation}
			where $r_h^0\in U_h^0$ and $\w_{h}^0\in \V_{h}^0$ is discrete  divergence-free. 
			And there exists $\w^0\in \bH_{0}(\curl; \Om)$ such that $\dive \w^0=0$, $\curl\w^0=\curl\Eh^0$,
			and
			\begin{equation}
				\|\w_{h}^0-\w^0\| \ls h\|\curl\Eh^0\|\ls  h^{\frac{1}{2}}\|\Eht\|_{\Ga} +h\|\curl\Eh\| ,
				\label{eq:stabwerr}
			\end{equation}
			where we have used \eqref{eq:ehcerr} to derive the last inequality. Moreover, 
			from \cite[Theorem 2.12]{amrouche1998vector}, we have 
			\eq{\label{eq:w0}
				\norm{\w^0}_1&\ls \norm{\w^0}+\norm{\curl\w^0}+\norm{\dive \w^0}\notag\\
				&\ls \norm{\w^0}+h^{-\frac{1}{2}}\|\Eht\|_{\Ga} +\|\curl\Eh\|\notag\\
				&\ls \norm{\w^0}+h^{-1}\|\Eh\|,}
			where we have used the inverse inequalities  $h^{-\frac12}\|\Eht\|_{\Ga},\|\curl\Eh\|\ls h^{-1}\|\Eh\|$ 
			to derive the last inequality. We have the following decomposition of $\|\E_h\|^2$:
			\begin{align}\label{eq:Ehdecomp}
				\|\E_h\|^2  & = (\Eh,\Eh) \notag \\
				& = (\Eh, \Eh-\Eh^0 )+(\Eh,\w_h^0-\w^0) +(\Eh,\nabla r_h^0)+(\Eh, \w^0 ).   
			\end{align}
			The first two terms on the right-hand side can be bounded using \eqref{eq:ehcerr} and 
			\eqref{eq:stabwerr}, respectively. To estimate the third term, we set $\bm v_h=\nabla r_h^0$ 
			in \eqref{eq:EEM} and using $\nabla r_h^0\in \bH_0( \curl ; \Omega) \cap \V_{h}$ and $\curl \nabla r_h^0 = 0$ to get
			\[
			a (\Eh,\nabla r_h^0) = -\ka^2(\Eh,\nabla r_h^0) = (\f,\nabla r_h^0),
			\]
			which implies
			\begin{align}
				\big|(\Eh,\nabla r_h^0)\big|  & \ls \ka^{-2}\|\f\|\|\nabla r_h^0\| \ls \ka^{-2}\|\f\|\|\Eh^0\|  \notag \\
				& \ls \ka^{-2}\|\f\|\left(\|\E_h\|+ h^{\frac{1}{2}}\|\Eht\|_{\Ga} \right).
				\label{eq:stabEhrh0}
			\end{align}
			Therefore, using the Cauchy's inequality and the  Young's inequality we obtain
			\begin{equation}
				\|\E_h\| \ls \|\Eh-\Eh^0\|+\| \w_h^0-\w^0\|+ \frac{1}{\ka^2} \|\f\| 
				+h^{\frac{1}{2}}\|\Eht\|_{\Ga} +\|\w^0\|.
				\label{eq:stabEhbound}
			\end{equation}
			Next we estimate the last term in \eqref{eq:stabEhbound}.
			
			{\it Step 3. Estimating $\|\w^0\|$ using a modified duality argument.} 
			Consider  the dual problem:   
			\begin{alignat}{2} 
				\curl\curl \ps -\ka^2\ps   &=\w^0 \qquad &&{\rm in }\  \Om,   \label{eq:stabDP1-1}\\
				\curl \ps\times \n +\ii \ka\la\ps_T   &= \bm{0} \qquad &&{\rm on }\  \Ga  \label{eq:stabDP1-2}.
			\end{alignat} 
			It is easy to verify that $\ps $ satisfies the following variational formulation:
			\begin{equation}
				a ( \bv,\ps)=( \bv,\w^0) \quad \forall \bv \in \V,
				\label{eq:stabdualVP1}
			\end{equation}
			and by Lemma \ref{thm:stability} and \eqref{eq:w0},  there hold the following estimates:
			\begin{align}
				\|\curl  \ps \|_{1}&+\ka\|\curl  \ps \| +
				\ka\|\ps \|_{1}+\ka^{2}\|\ps \|  +\ka^{2}\|\ps_T \|_{\Ga}   \lesssim \ka \| \w^0\|,  \label{eq:stabDP1} \\
				\|\ps\|_{2} &\ls \ka \|\w^0\| +\ka^{-1}  \|\w^0\|_1\ls    \ka \|\w^0\| +\ka^{-1}h^{-1}\|\Eh\|.  \label{eq:stabDP2}
			\end{align} 
			{For simplicity, we suppose that $k^3h^2\ls 1$.}	Using \eqref{eq:proj}, 
			\eqref{eq:stabdualVP1}--\eqref{eq:stabDP2}, and Lemma~\ref{lem:Pherror},   
			we deduce that
			\begin{align}
				\big|	( \Eh,\w^0) \big|
				&=\big| a ( \Eh , \ps) \big|  =\big| a ( \Eh , \pps) +a ( \Eh , \ps-\pps)\big| \notag \\
				&=\big| (\f, \pps)+ \bra \g , (\pps)_T \ket +  a ( \Eh , \ps-\pps) \big|  \notag \\
				&=\big| (\f, \ps)+ \bra \g , \ps_T \ket +  (\f, \pps - \ps )+ 
				\bra \g , (\pps - \ps )_T \ket - 2\ka^2(\Eh,\ps-\pps) \big|  \notag \\ 
				&\ls  \|\f\|\big(\|\ps\|+h^2\|\ps\|_2\big) +\|\g\|_{\Ga}\big(\|\ps_T\|_{\Ga} 
				+ h^{\frac{3}{2}}\|\ps\|_2 \big) +\ka^2 h^2 \|\E_h\| \|\ps\|_2           \notag \\ 
				&\ls \ka^{-1}  \|\f\|\|\w^0\| +\ka h^2 \|\f\|\|\w^0\| +\ka^{-1}h\|\f\|\| \Eh\|            \notag \\ 
				&\phantom{{}+} +  \ka^{-1} \|\g\|_{\Ga} \|\w^0\| 
				+\ka h^{\frac{3}{2}}\|\g\|_{\Ga} \|\w^0\| +\ka^{-1}  h^{\frac{1}{2}}\|\g\|_{\Ga}  \| \Eh\|           \notag \\ 
				&\phantom{{}+} +  \ka^3 h^2\|\E_h\|\|\w^0\| +\ka h \|\E_h\|^2       \notag \\ 
				&\ls  \ka^{-1}\big(\|\f\|+\|\g\|_{\Ga}\big)\|\w^0\|+ \ka^3 h^2\|\E_h\| \|\w^0\|             \notag \\ 
				&\phantom{{}+} +   \ka^{-1}h^\frac12\big(\|\f\|+  \|\g\|_{\Ga}\big)\|\E_h\| + \ka h \|\E_h\|^2.
				\label{eq:stabehw}
			\end{align}
			Note that $(\w^0,\nabla r_h^0)=0$. We have
			\eqn{\|\w^0\|^2 & =(\w^0+\Eh-\Eh +\w_{h}^0 +\nabla r_h^0 -\w_{h}^0, \w^0) \notag \\
				& = (\w^0-\w_{h}^0, \w^0)  +(\Eh^0-\Eh, \w^0) +(\Eh, \w^0),      }
			which together with \eqref{eq:stabehw} implies that
			\eqn{
				\|\w^0\|^2&\ls\|\w^0-\w_{h}^0\|^2+ \|\Eh^0-\Eh\|^2+\ka^{-1}\big(\|\f\|+\|\g\|_{\Ga}\big)\|\w^0\|
				+ \ka^3 h^2\|\E_h\| \|\w^0\|    \notag \\ 
				&\phantom{{}+}+   \ka^{-1}h^\frac12\big(\|\f\|+  \|\g\|_{\Ga}\big)\|\E_h\| + \ka h \|\E_h\|^2.  }
			Therefore, by the Young's inequality we have
			\eq{\label{stabw0wEstimate}
				\|\w^0\|&\ls\|\w^0-\w_{h}^0\|+ \|\Eh^0-\Eh\|+\ka^{-1}\big(\|\f\|+\|\g\|_{\Ga}\big)+ \big(\ka^3 h^2+(\ka h)^\frac12\big)\|\E_h\|.}

			{\it Step 4.  Summing up.} By plugging \eqref{stabw0wEstimate} into \eqref{eq:stabEhbound} and 
			using \eqref{eq:ehcerr} and \eqref{eq:stabwerr}, we obtain
			\eqn{
				\|\E_h\| &\ls h^{\frac{1}{2}}\|\Eht\|_{\Ga}+h\|\curl\Eh\|+ \ka^{-1}\big(\|\f\|+\|\g\|_{\Ga}\big)
				+\big(\ka^3 h^2+(\ka h)^\frac12\big)\|\E_h\|, }
			which together with \eqref{eq:stabbd}--\eqref{eq:stabEner} gives
			\begin{align*}
				\|\E_h\| & \ls \big(\ka^{-1}+\ka^{-1}h^\frac12+k^{-\frac12}h\big)\big(\|\f\|+\|\g\|_{\Ga}\big)
				+ \big(\ka^3 h^2+(\ka h)^\frac12\big)\|\E_h\|.
			\end{align*} 
			Therefore, there exists a constant $C_0>0$  such that if  $\ka^3 h^2\le C_0$,  then
			\begin{align}
				\|\E_h\| \ls \ka^{-1}\big(\|\f\|+\|\g\|_{\Ga}\big).
				\label{eq:stabl2}
			\end{align} 
			Then \eqref{eq:EEMStability} follows from \eqref{eq:stabl2}, \eqref{eq:stabbd} and \eqref{eq:stabEner}. 
			This completes the proof of the theorem.	
		\end{proof}
		
		\begin{remark}\label{rem:EEMwellposed}
			{\rm (a)}{ This stability  estimate \eqref{eq:EEMStability}  of the  EE  solution is  of the same  order as that 
				of the continuous solution (cf. \eqref{eq:stability0} in Lemma~\ref{thm:stability} )  and still holds when $\dive\f\neq 0$}. 
			
			{\rm (b)} {The modified duality argument used in Step 3 differs from the standard one only by replacing 
				the interpolation of the solution ($\ps$) to the dual problem used there by its elliptic projection 
				($\pps$). If the standard duality argument is used instead, only asymptotic stability estimate under
				the mesh condition $\ka^2 h\le C_0$ can be obtained. The modified duality argument was first used to 
				derive preasymptotic error estimates \cite{ZhuWu2013II,DuWu2015} and preasymptotic stability estimates 
				\cite{Wu2018jssx} for FEM and CIP-FEM  for the Helmholtz equation with large wave number.	}
		\end{remark}

		\subsection{Error estimates of EEM}

		The following theorem gives the main result of this paper.
		
		\begin{theorem}\label{thm:L2result}  
			Assume that $\Omega \subset \mathbb{R}^{3}$ is a bounded $C^{2}$-domain and strictly star-shaped with respect to  
			a point ${\bm x}_{0} \in \Om$ and that $\dive\f=0$ in $\Om$ and $\g\cdot\n=0$ on $\Ga$. Let $\E$ be the solution 
			to the problem \eqref{eq:eq}--\eqref{eq:bc} and $\E_h$ be the EEM solution to \eqref{eq:EEM}.
			Then there exists a constant $C_0>0$ independent 
			of $\ka$ and $h$ such that when  {$\ka^3 h^2\le C_0$},  
			the following preasymptotic error estimates hold:
			\eq{\norme{\E-\Eh} &\ls   \big(\ka h+ \ka^3 h^2\big) \Mfg,	\label{eq:energyResult}\\
				\ka\|\E-\Eh\|  &\ls \big( {(\ka h)^2 + \ka^3 h^2}\big) \Mfg. \label{eq:L2result}
			}
		\end{theorem}
		\begin{proof}
			Denote by $\bbeta =\E-\PEp, {\bxi}_h=\Eh - \PEp$ , then $ \E-\Eh =\bbeta-{\bxi}_h $ and 
			\begin{align*} 
				a ({\bxi}_h,\bv_{h}) &= a (\bbeta, \bv_{h})  = -2 \ka^2 (\bbeta, \bv_{h})\quad {\forall \bv_h\in\V_h.}
			\end{align*} 
			{From Lemmas~\ref{thm:stability}--\ref{lem:Pherror}, we have}
			\begin{equation*}
				\| \bbeta\|  \ls h^2 \|\E\|_2   \ls \ka h^2 \Mfg,\quad \norme{\bbeta} \ls \ka h \Mfg.
			\end{equation*}
			By using Theorem~ \ref{thm:EEMStability}  {(with $\bm f=-2\ka^2\bbeta$ and $\g=0$)}, we obtain
			\begin{align*} 
				\norme{\bxi_h}+ \ka\|{\bxi}_h\|  & \ls \ka^2 \| \bbeta\|    \ls \ka^3 h^2 \Mfg.  
			\end{align*} 	 
			Combining the above two estimates and using the triangle  {inequality completes the proof of the theorem.}
		\end{proof}	
		
		{\begin{remark}\label{rem:EEM-error-estimates}
				{\rm (a)}  {\it A priori} error estimates of the  EEM for the time-harmonic Maxwell equations with  impedance 
				boundary conditions have been provided in \cite{monk2003,gatica2012finite}.
				However, explicit dependence on the wave number $\ka$ is not discussed in either study. 
				
				{\rm (b)} When $\dive\f\neq 0$ and $\Om$ is $C^3$, the estimates \eqref{eq:energyResult}--\eqref{eq:L2result} still 
				hold if an additional term $\ka^{-3}\|\dive\f\|_1$ is added to $\Mfg$ (see Remark~\ref{rem3.1}(c)). 
				
				{\rm (c)} To the best of the authors' knowledge, this theorem gives the first preasymptotic error 
				estimates in both $\bH(\curl)$ and $\bL^2$ norms for the second type N\'{e}d\'{e}lec linear EEM  
				for Maxwell's equations, under the mesh condition that $k^3h^2$ is sufficiently small. 
				\cite{Melenk_2020,melenk2023wavenumber} proposed asymptotic error estimates for the first type
				N\'{e}d\'{e}lec $hp$-EEM for the Maxwell's equations with transparent and impedance boundary
				conditions, respectively, which say that the discrete solution is pollution-free if $p \gtrsim  \ln  \ka$
				and $\ka h/p$ is sufficiently small. However, the lowest order EEM was excluded in both works. 
				For the second order EEM, it requires that $k^8h$ is sufficiently small 
				(see \cite[Remark 4.19]{Melenk_2020}) for the case of transparent boundary condition or
				$k^2h$ is sufficiently small  for the case of impedance boundary condition. 
				
				{\rm (d)} 
				Our error estimation process is to prove the stability of the EEM first, and then use it to derive the error estimates. 
				We had tried to derive preasymptotic error estimates for the EEM mimicking the usual process (see, e.g., \cite{monk2003}),
				including using the modified duality argument to derive the error estimates directly,  but failed. 
		\end{remark}}
		
		\section{ Preasymptotic  error estimates for CIP-EEM }\label{sc5}

		In this section, we show that the  error estimates  \eqref{eq:energyResult}--\eqref{eq:L2result}
		also hold for CIP-EEM with general   {penalty} parameters, which can be proved by following the lines in Section \ref{sc4}.
		We omit the similarities and  point out only the differences.
		
		{We first introduce the energy space:
			\eqn{\hat\V:=\big\{\bv\in\V:\; (\curl\bv)|_K\in \bH^1(K) \quad \forall K\in\Th\big\},}
			and define the energy norm:}
		\begin{equation}
			\energy{\bv}:=\Big( \norme{\bv}^2+\sum_{\fa  \in  {\mathcal{F}}_{h}^{I}} |\gamma_{\fa } | h_{\fa }\|
			\jv{\curl  \bv }  \|_{0, \fa }^{2} \Big)^{\frac{1}{2}}.
			\label{eq:energyNormDefine}
		\end{equation}
		Similar to \eqref{eq:hat a}--\eqref{eq:proj}, we define
		\eq{\label{eq:hat a gamma}
			{\hat a}_{\gamma}({\bm u}, {\bm v}):= a_{\gamma}({\bm u}, {\bm v} )+2\ka^2({\bm u}, {\bm v})  \quad \forall {\bm u}, {\bm v} \in \hat\V, } 
		{and introduce} $\bH(\curl)$-elliptic projections ${\bm P}_h^\pm$ onto $\V_h$ as
		\begin{align}
			\hat{a}_{\gamma}\big({\bm u}-{\bm P}_h^+{\bm u},{\bm v}_h\big)=0, \quad
			\hat{a}_{\gamma}\big({\bm v}_h,{\bm u}-{\bm P}_h^-{\bm u}\big)=0  \quad \forall {\bm v}_h\in\Vh. \label{eq:CIPproj} 
		\end{align}
		
		The following lemma provides the continuity and coercivity of the sesquilinear form $\hat a_\ga$.  
		\begin{lemma}\label{lem:ahat}
			There exists a positive number $\alpha_0$ such that if $\re\gamma_{\fa } \geq -\alpha_0 $ 
			and $\im\gamma_{\fa }\leq 0$, then
			\begin{subequations}
				\begin{align}
					\big|\hat{a}_{\gamma} (\bv,\w ) \big|   
					&\ls \energy{\bv}\energy{\w}    \quad \forall \bv,\w\in  {\hat\V} , \; \label{eq:discontinuity}\\
					\re \hat{a}_{\gamma} (\bv_h,\bv_h ) -\im \hat{a}_{\gamma} (\bv_h,\bv_h ) 
					&\gtrsim \energy{\bv_h}^2   \quad \forall \bv_h \in \V_h.   \; \label{eq:discoercivity} 
				\end{align}
			\end{subequations} 	
		\end{lemma}
		\begin{proof}
			{Since the proof of \eqref{eq:discontinuity} is straightforward} by using the Cauchy-Schwarz inequality, we only prove
			\eqref{eq:discoercivity}.
			First,  {for any $\bv_h \in \V_h$ we have}
			\eqn{
				\re \hat{a}_{\gamma} (\bv_h,\bv_h ) = \|\curl \bv_h \|^2 +\ka^2 \|\bv_h\|^2 
				+\sum_{\fa  \in\mathcal{F}_h^I} \re(\ga_{\fa })	h_\fa 
				\bra \jv{\curl \bv_h},\jv{\curl\bv_h} \ket_{\fa }.}
			By using the  {local} trace inequality on each element and the inverse inequality, there exists 
			$\beta_0 >0$ such that
			\eqn{ \sum_{\fa  \in\mathcal{F}_h^I}  	h_\fa 
				\bra \jv{\curl \bv_h},\jv{\curl\bv_h} \ket_{\fa }
				\leq \beta_0 \|\curl \bv_h \|^2.}
			Thus, by taking $\alpha_0 = \frac{1}{2\beta_0}$  and letting $\re\gamma_{\fa } \geq -\alpha_0 $, we conclude that
			\begin{equation}
				\re \hat{a}_{\gamma}(\bv_h,\bv_h ) \geq \frac{1}{2} \|\curl \bv_h\|^2 + \ka^2\|\bv_h\|^2.
				\label{eq:rehaCoer}
			\end{equation}
			Secondly, it is obvious that 
			\begin{equation}
				- \im	\hat{a}_{\gamma} (\bv_h,\bv_h ) \geq  \ka \la \|\bv_{h,T}\|_{\Ga}^2.
				\label{eq:imhaCor}
			\end{equation}
			Combining \eqref{eq:rehaCoer} and \eqref{eq:imhaCor}, we obtain \eqref{eq:discoercivity}. This completes the proof of the lemma.
		\end{proof}
		
		The following lemma gives error estimates for the $\bH(\curl)$-elliptic projections.
		\begin{lemma}\label{lem:PEperror} Let $\gamma:= \max_{\fa  \in\mathcal{F}_h^I}  |\gamma_{\fa }|$. Suppose  $\bm u\in\bH^2(\Om)$ and $\gamma_{\fa }$  satisfies $\re\gamma_{\fa } \geq -\alpha_0 $ and	$\im\gamma_{\fa }\leq 0$. Then
			\begin{subequations}
				\begin{align}
					\energy{ {\bm u}-\boldsymbol{P}^{\pm}_h{\bm u}}&\ls \inf_{{\bm u}_h\in\V_h}\energy{{\bm u}-{\bm u}_h}
					\ls (1+\ga)^{\frac{1}{2}}h  |{\curl \bm u}|_{1}  +C_{\ka h}(\ka h)^{\frac{1}{2}}h |{\bm u}|_{2},\label{eq:PEperror_energy}\\
					\|{\bm u}-\boldsymbol{P}^{\pm}_h{\bm u}\|    &\ls C_{\ka h}(1+\ga) h^{2 } |{\bm u}|_{2}, \label{eq:PEperror_l2} \\
					\|({\bm u}-\boldsymbol{P}^{\pm}_h{\bm u})_T\|_{\Ga} & \ls C_{ \ka h}(1+\ga)h^{\frac{3}{2}}  |{\bm u}|_{2}\label{eq:boundPEerror}.
				\end{align}
			\end{subequations} 	
		\end{lemma}
		\begin{proof}
			This lemma has been essentially demonstrated  in \cite[Theorem~4.3, Theorem~4.4]{pplu2019}. 
			To establish the proof of this lemma, it is sufficient to substitute  {\cite[Lemma~4.1]{pplu2019} with the above} Lemma~\ref{lem:ahat}.
		\end{proof}
		
		Using Lemmas \ref{lem:ahat} and \ref{lem:PEperror}, following the same lines of the proof in 
		Theorem~\ref{thm:EEMStability}, we obtain the following stability estimate for CIP-EEM.
		\begin{theorem}\label{thm:CIPStability}
			Let $\E_h$ be the solution to \eqref{eq:disvp}, $\f\in \bL^2(\Om)$, $\g\in\bL^2(\Ga)$. 
			There exist $C_0 >0$ , $\alpha_0>0$ such that
			if $\ka^3 h^2 \leq C_0$, $\re\gamma_{\fa } \geq -\alpha_0 $, 
			$\im\gamma_{\fa }\leq 0$, and $|\gamma_{\fa }|\ls 1$, then
			\begin{equation}
				\|\curl\E_h\| +\ka \|\E_h\|+\ka \|\E_{h,T}\|_{\Ga} {+\Big(\sum_{\fa  \in  {\mathcal{F}}_{h}^{I}} |\gamma_{\fa } | h_{\fa }\|
					\jv{\curl  \bv }  \|_{0, \fa }^{2} \Big)^{\frac12}} \ls \|\f\| +  \|\g\|_{\Ga}.
				\label{eq:CIPstability}
			\end{equation} 
			{And as a consequence, the CIP-EEM \eqref{eq:EEM} is well-posed.}
		\end{theorem}
		
		\begin{remark}\label{rem:CIPwellposed}
			{\rm (a)} {This stability bound of the  CIP-EEM  solution is also of the same  order as 
				that of the continuous solution (cf. Lemma~\ref{thm:stability}) and holds when $\dive\f\neq 0$}.
			
			{\rm (b)}  { When $\im{\gamma_{\fa } } < 0 $, it is proved in \cite[Theorem~3.1]{pplu2019} 
				that the CIP-EEM  is absolutely stable (i.e. stable without any mesh constraint). Moreover, 
				by using a trick of  ``stability-error iterative improvement" developed in \cite{fw2011}, 
				it is proved in \cite[Theorem~4.6]{pplu2019} that the CIP-EEM satisfies the stability estimate
				$\energy{\E_h} \ls\Mfg $ under the conditions $\ka^3 h^2 \leq C_0$, $\re\gamma_{\fa }= 0 $, and
				$ -\im\gamma_{\fa } \simeq 1 $. Clearly, our result gives an improvement of that in \cite{pplu2019} 
				when  $\ka^3 h^2 \leq C_0$.  When $\im{\gamma_{\fa } } \geq 0 $ and $\ka^3 h^2$ is 
				large, the well-posedness of the   CIP-EEM (including EEM)    is still open.}	
		\end{remark}
		
		The following Theorem gives  {preasymptotic} error estimates for CIP-EEM, the proof of which is 
		{similar to that of Theorem~\ref{thm:L2result} and is omitted here.}
		
		\begin{theorem}\label{thm:CIPL2result} 
			Assume that $\Omega \subset \mathbb{R}^{3}$ is a bounded $C^{2}$-domain and strictly star-shaped with respect to  
			a point ${\bm x}_{0} \in \Om$ and that $\dive\f=0$ in $\Om$ and $\g\cdot\n=0$ on $\Ga$. Let $\E$ be the solution to the 
			problem \eqref{eq:eq}--\eqref{eq:bc} and $\E_h$ be the CIP-EEM solution to \eqref{eq:disvp}.
			Then there exist $C_0 >0$ , $\alpha_0>0$ such that
			when $\ka^3 h^2 \leq C_0$, $\re\gamma_{\fa } \geq -\alpha_0 $, 
			$\im\gamma_{\fa }\leq 0$, and $|\gamma_{\fa }|\ls 1$, 
			the following estimates hold:
			\eq{\energy{\E-\Eh} &\ls   \big(\ka h+ \ka^3 h^2\big) \Mfg,	\label{eq:CIPenergyResult}\\
				\ka\|\E-\Eh\|  &\ls \big((\ka h)^2 + \ka^3 h^2\big) \Mfg. \label{eq:CIPL2result}
			}
		\end{theorem}

		\begin{remark}\label{rem:CIPerror}
			{\rm (a)} An important reason for studying the CIP-EEM is its potential to reduce the pollution 
			error by tuning the penalty parameters. 
			
			{\rm (b)} Remark~\ref{rem:EEM-error-estimates}(b) also holds for the CIP-EEM.
			
			{\rm (c)} The error estimate in the energy norm was also given in \cite[Theorem~4.6]{pplu2019} 
			under the conditions $\ka^3 h^2 \leq C_0$, $\re\gamma_{\fa }= 0 $, and $-\im\gamma_{\fa } \simeq 1 $.
			Our results relax the conditions on the penalty parameters and give a new $\bL^2$ error estimate 
			for the CIP-EEM. Such a relaxation  is meaningful in practical computations, since the penalty 
			parameters that can significantly reduce the pollution error are usually close to negative real numbers (see the next section).
		\end{remark}

		\section{Numerical example }\label{sc6}
		
		In this section, we report an example to verify the theoretical findings and to demonstrate the 
		performance of the EEM and the CIP-EEM,  in particular, the potential of the CIP-EEM to significantly
		reduce the pollution errors by selecting appropriate penalty parameters. 
		The  N\'{e}d\'{e}lec's linear edge elements of the  second type  are used in the numerical tests. 
		The linear systems resulted from edge element discretizations are solved by \texttt{pardiso}\cite{pardiso},
		which is commonly used to solve large, sparse systems of linear equations.
		
		\begin{example}\label{Ex1} We simulate the following three-dimensional time-harmonic Maxwell  problem:
			\begin{equation*}
				\begin{cases}
					\curl\curl\E-k^2\E  = \f, \quad &{\rm in}\quad \Om:={(1,2)}^3,\\
					\curl \E \times  \n -\ii\ka \E_T=\g,   \quad &{\rm on}\quad \Ga := \pa  \Omega,  \\
				\end{cases}
			\end{equation*}
			and $\f$ and $\g$ are chosen such that the
			exact solution is given by
			\eqn{\E= \ka \sum_{m=-1}^{1} h_1^{(1)}(\ka  r) \nabla_{S} Y_1^{m}\times \hat{\br} 
				+ \frac{1}{\ka}\big(  {\sin}(\ka z), \sin(\ka y), \sin(\ka x) \big),} 
			where $ h_1^{(1)}$ is the  spherical Hankel function of the first kind of  order $1$, $\nabla_{S}$ is the 
			surface gradient operator (see, e.g., \cite{monk2003}),  $Y_1^{m}, m=-1,0,1$,  
			are the   spherical harmonics of order $1$  on the unit sphere (see, e.g., \cite{colton1998inverse}),  
			$\br = (x,y,z), r= \abs{\br} $, and $\hat{\br} = \br/ r$.
		\end{example}
		
		In this numerical example, we triangulate the computational domain $\Om$ into a mesh of the type
		cub6 \cite{Naylor1999} (also called Cube-VI-II in \cite[Chapter~4]{zyphd}). Specifically, we first divide $\Om$
		into small cubes of the same size, then divide each small cube into $6$ small tetrahedrons as 
		shown in Figure~\ref{fig:faces} (left).  We note that such a mesh can be simply generated
		by \texttt{delaunayTriangulation} in MATLAB on a uniform Cartesian grid.
		
		The penalty parameters for the CIP-EEM are simply taken from the parameters used by the CIP-FEM for the Helmholtz equation \cite[Chapter~4]{zyphd}, 
		but with pure imaginary perturbations to enhance stability, i.e., we set  
		\eq{\label{gamma_f}\ga_{\fa }=\ga_1 := -\frac{\sqrt{2}}{24}-0.01\ii, \quad \ga_{\fa }=\ga_2:=-\frac{\sqrt{6}}{72} -0.01\ii, \quad\text{and}\quad 
			\ga_{\fa }=\ga_3:=-\frac{\sqrt{2}}{48} -0.01\ii} 
		according to three different kinds of interior faces as indicated in Figure~ \ref{fig:faces}. 
		We would like to remark that the real  parts $\big\{-\frac{\sqrt{2}}{24},-\frac{\sqrt{6}}{72},-\frac{\sqrt{2}}{48}\big\}$
		of the penalty parameters are obtained by a \emph{dispersion analysis} for the CIP-FEM, which is an essential tool
		to understand the  dispersive behavior of numerical schemes, 
		and it is commonly believed that the  pollution errors are of the same order as the phase difference between the exact and numerical solutions \cite{Ainsworth2004,ainsworth2004dispersive,Zhou2022Dispersion,monkdispersion,monk2003,burman1D}. 
		Such real penalty parameters can reduce phase difference of the CIP-FEM for the Helmholtz equation 
		on cub6-type meshes from $\bO(\ka^3 h^2)$ to $\bO(\ka^5 h^4)$ \cite[Chapter~4]{zyphd}, and numerical tests
		there do show that the pollution errors can also be significantly reduced. Here we expect that the penalty
		parameters in \eqref{gamma_f} can significantly reduce the pollution errors of the CIP-EEM. We postpone 
		the systematic dispersion analysis of CIP-EEM for Maxwell equations  to a future work, and only provide
		some numerical tests to confirm this expectation.

		\begin{figure}
			\begin{center}
				\includegraphics[scale=0.9,trim={0cm 1cm 0cm 0}]{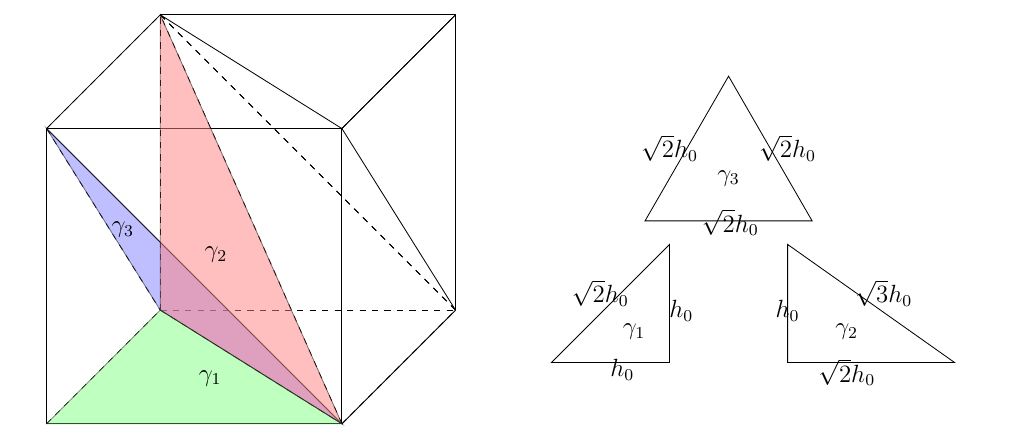}
				\vspace{1em}
			\end{center}
			\caption{Mesh of type cub6 and penalty parameters  on different interior faces.}
			\label{fig:faces}
		\end{figure}

		\begin{figure}
			\centering
			\includegraphics[scale=0.55,trim={2cm 0 2cm 0}]{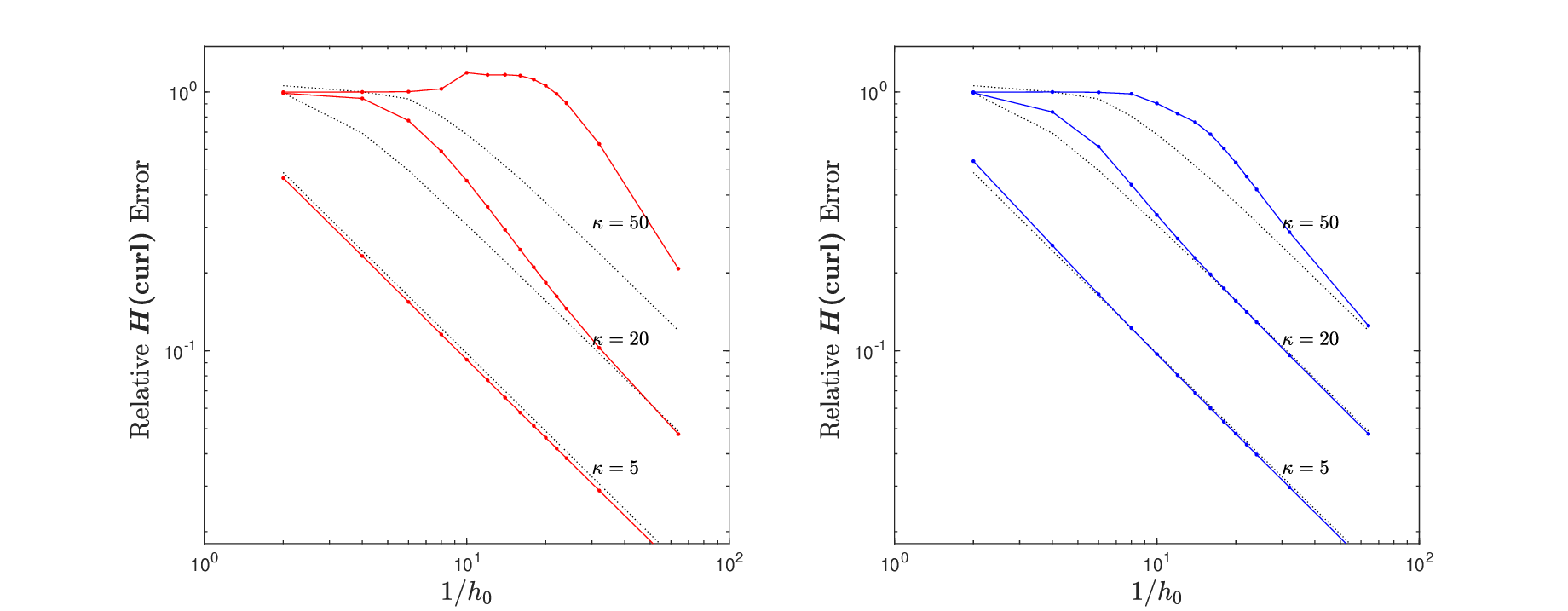}\\
			\caption{Log-log plots of  the relative $\bH(\curl)$ errors of  the   EE solution (left), the CIP-EE solution
				(right), and the interpolation (black dotted line)  versus $1/h_0$ with $1/h_0= 2, 4,\cdots, 24, 32 $ 
				and $64$,  for $\ka  = 5, 20$,  and   $50 $, respectively. 
				\label{fig:fixK}
			}
		\end{figure}
		\begin{figure}
			\centering
			\includegraphics[scale=0.55,trim={2cm 0 2cm 0}]{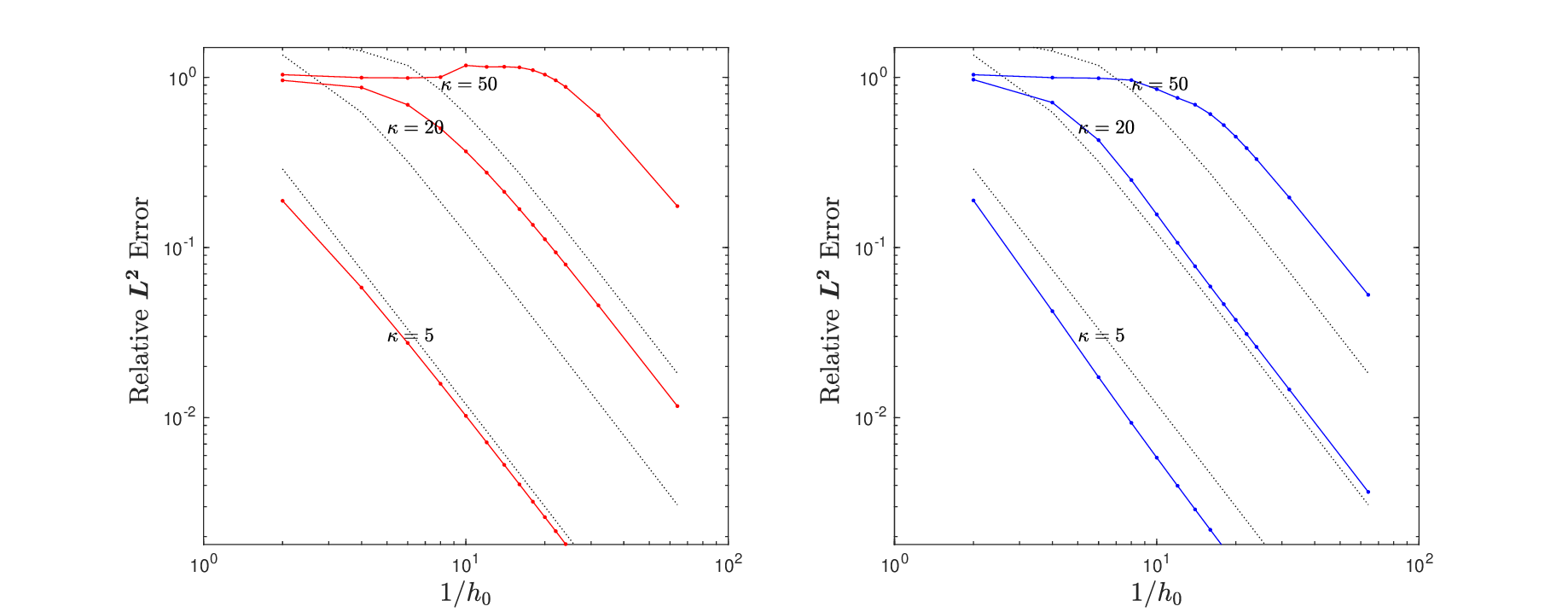}\\
			\caption{Log-log plots of the relative $\bL^2$ errors of  the  EE solution (left), the CIP-EE solution (right),
				and the interpolation (black dotted line)  versus $1/h_0$ with $1/h_0= 2, 4,\cdots, 24, 32 $ and $64$, 
				for $\ka  = 5, 20$,  and   $50 $, respectively. 
				\label{fig:fixKl2}
			}
		\end{figure}

		Figures~\ref{fig:fixK} and \ref{fig:fixKl2} illustrate the relative $\bH(\curl)$ errors  and  the relative $\bL^2$
		errors   of the EE solutions (left), the
		CIP-EE solutions (right),  and the interpolations (black dot line) for $\ka = 5$, $20$, and $50$, respectively.
		It demonstrates that when $\ka = 5$, the errors of the solutions to EEM   and  CIP-EEM  closely match those of
		the corresponding EE interpolations, implying the absence of pollution errors for small wave numbers. 
		Conversely, for large $\ka$, the relative errors of the  EE  solutions decay slowly, starting from a point
		considerably distant from the decaying point of the corresponding EEM interpolations. This behavior vividly 
		exposes the presence of pollution errors in the  EEM. The  CIP-EE  solutions exhibit 
		a similar behavior to the EE solutions, but the pollution range of the former is significantly smaller 
		than that of the latter.

		Figure~\ref{fig:ReEnerplot} presents plots of relative $\bH(\curl)$ errors   
		with mesh constraint $\ka h_0=1$ for   EEM   and  CIP-EEM  with $\ka=1,2,\cdots , 100 $, respectively. 
		Note that  for small wave number $\ka$, the errors of EE and CIP-EE solutions closely match those 
		of the corresponding EE interpolations, implying that pollution errors do not manifest for small 
		wave numbers. For large values of $\ka$, the relative errors of EE solutions deteriorate rapidly. 
		This behavior clearly demonstrates the impact of the pollution error in the EEM. 
		The CIP-EE solutions behave well till $\ka=100$, which shows the pollution effect for this method is 
		significantly smaller than that of the EEM.
		
		\begin{figure}[h]
			\centering
			\includegraphics[scale=0.65,trim={2cm 0 2cm 0}]{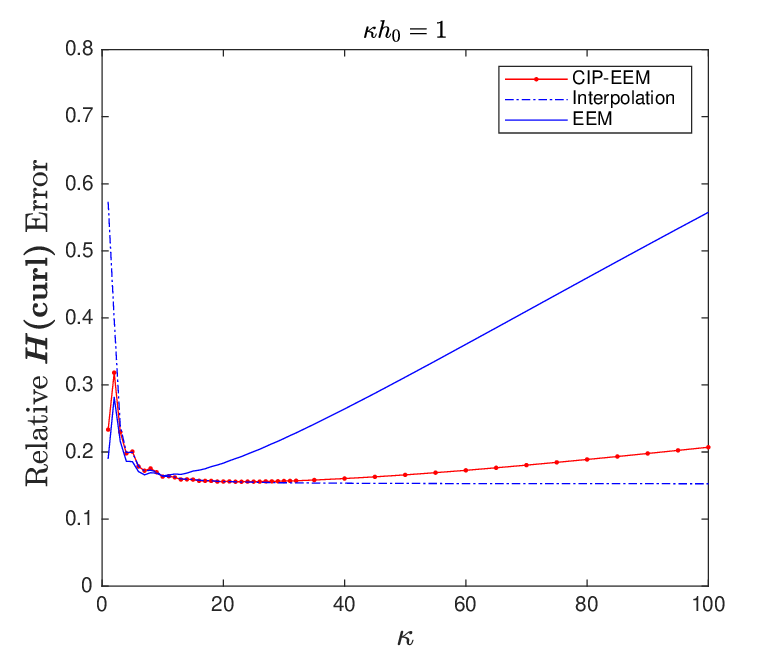}\\
			\caption{Relative $\bH(\curl)$ errors  with $\ka h_0=1$.
				\label{fig:ReEnerplot}
			}
		\end{figure}

		\appendix
		\renewcommand{\theequation}{A.\arabic{equation}}
		\renewcommand{\thetheorem}{A.\arabic{theorem}}
		\renewcommand{\thelemma}{A.\arabic{lemma}}
		\renewcommand{\thefigure}{A.\arabic{figure}}
		\setcounter{equation}{0}
		\setcounter{theorem}{0}
		\setcounter{lemma}{0}
		\setcounter{figure}{0}
		
		\section*{Appendix: Proof of Lemma~\ref{lem:Pherror}} 
		
		Denote by $\pi_N$ the interpolation onto the second-type N\'ed\'elec
		edge element space $\V_h$. The following lemma gives the interpolation error estimates.
		\begin{lemma}\label{IE} We have 
			\begin{align}
				\label{approx1}
				& \| \curl ({\bm u}-\pi_N{\bm u} ) \| \lesssim h|{\curl \bm u}|_{1},  \\
				\label{approx2}
				&\|{\bm u}-\pi_N{\bm u}\|\lesssim h^2|{\bm u}|_{2},\\
				\label{approx3}
				&\|({\bm u}-\pi_N{\bm u})_T\|_{\Ga}\lesssim h^\frac32|{\bm u}|_2,\\
				&\interleave{\bm u}-\pi_N{\bm u}\interleave\lesssim h|{\curl\bm u}|_{1}+C_{\ka h} (\ka h)^{\frac{1}{2}}h  |{\bm u}|_{2}.\label{approx4a}
			\end{align}
		\end{lemma}
		\begin{proof}
			For \eqref{approx1}--\eqref{approx2}, we refer to \cite[Theorems 8.15, 5.41]{monk2003}. The estimate \eqref{approx3} follows from the local trace inequality and  \eqref{approx4a} follows from \eqref{approx1}--\eqref{approx3}. The proof is completed.
		\end{proof}
		
		Without loss of generality, we prove Lemma~\ref{lem:Pherror} only for $\tilde{{\bm u}}_h:=P_h^+\bm u$. The proof of \eqref{eq:Pherror_energy} is obvious and is omitted. Next we prove \eqref{eq:Pherror_l2} and \eqref{eq:boundPherror}. 
		
		The following lemma gives an estimate of $\|({\bm u}-\tilde{\bm u}_h)_T\|_{\Ga}$.
		
		\begin{lemma}\label{lem:bd}
			It holds
			\begin{align}
				\label{trace}
				\|({\bm u}-\tilde{\bm u}_h)_T\|_{\Ga}\ls \|({\bm u}-{\bm v}_h)_T\|_{\Ga}+C_{\ka h}\big( h^{-\frac12}\|{\bm u}-{\bm v}_h\| +h^\frac12\interleave{\bm u}-{\bm v}_h\interleave\big)   \quad \forall \bm v_h\in \V_h.
			\end{align}
		\end{lemma}
		\begin{proof}
			Let ${\bm \Phi}_h:=\tilde{{\bm u}}_h-{\bm v}_h\in \V_h$. Similar  to Lemma \ref{lem:HD}, we have the following decompositions (see, e.g., \cite[Remark~3.46, Lemma~7.6]{monk2003}):
			\begin{align}
				\label{decom2}
				{\bm \Phi}_h=\nabla r+\w=\nabla r_h+\w_h,
			\end{align}
			where $r\in H^1(\Omega)$, $r_h\in U_h$, $\w\in \bH^1(\Omega)$, and $\w_h\in \V_h$, $\w$ is divergence-free in $\Om$, $\w\cdot\boldsymbol{\nu}=0$ on $\Ga $, and there also holds
			\begin{align}
				\label{wd}
				\|\w-\w_h\| \lesssim h\|{\curl}\ {\bm \Phi}_h\| \lesssim h\interleave{\bm u}-{\bm v}_h\interleave,
			\end{align}
			where we have used \eqref{eq:Pherror_energy} to derive the last inequality. 
			
			Next, we   establish  a relationship between $\|({\bm u}-\tilde{\bm u}_h)_T\|_{\Ga}$ and $\|\w\| $. Denote by 
			\begin{align*}
				d({\bm u},\bm v):=\ka^2({\bm u},\bm v)-{\ii} \ka\lambda \langle {\bm u}_T,\bm v_T\rangle.
			\end{align*}
			From \eqref{eq:hat a} and \eqref{eq:proj}, we have
			\begin{align}\label{drh}
				d({\bm u}-\tilde{\bm u}_h,\na r_h)=\hat a ({\bm u}-\tilde{\bm u}_h,\na r_h)=0,
			\end{align}
			which implies
			\begin{align*}
				d({\bm u}-\tilde{\bm u}_h,{\bm u}-\tilde{\bm u}_h)=d({\bm u}-\tilde{\bm u}_h,{\bm u}-{\bm v}_h-\w_h).
			\end{align*}
			And hence from the Cauchy's inequality we obtain
			\begin{align*}
				\big|d({\bm u}-\tilde{\bm u}_h,{\bm u}-\tilde{\bm u}_h)\big|\ls \big|d({\bm u}-{\bm v}_h-\w_h,{\bm u}-{\bm v}_h-\w_h)\big|,
			\end{align*}
			which gives
			\begin{align*}
				\ka\la\|({\bm u}-\tilde{\bm u}_h)_T\|_{\Ga}^2\le \ka\la\|({\bm u}-{\bm v}_h)_T\|_{\Ga}^2+\ka\la\|\w_{h,T} \|_{\Ga}^2+\ka^2\|{\bm u}-{\bm v}_h\| ^2+\ka^2\|\w_h\| ^2.
			\end{align*}
			Therefore, by noting $\| \w_{h,T} \|_{\Ga}\ls h^{-\frac12}\|\w_h\|$ and using \eqref{wd}, we conclude that
			\begin{align}\label{lb}
				\|({\bm u}-\tilde{\bm u}_h)_T\|_{\Ga}\ls& \|({\bm u}-{\bm v}_h)_T\|_{\Ga}+\ka^\frac12\|{\bm u}-{\bm v}_h\|\notag \\
				&+C_{\ka h}\big(h^{-\frac12}\|\w\| +h^\frac12\interleave{\bm u}-{\bm v}_h\interleave\big).
			\end{align}
			
			We utilize the duality argument to estimate $\|\w\|$, first we begin by introducing
			the dual problem:
			\begin{alignat}{2}
				\label{D-equations}
				{\curl}\,{\curl}\,{\bm z}+\ka^2 {\bm z}&=\w \qquad &&{\rm in }\ \Omega,\\
				\label{D-boundary}{\curl}\,{\bm z}\times \boldsymbol{\nu}+{\ii}\ka \lambda {\bm z}_T&= \bm{0} \qquad &&{\rm on }\
				\Ga,
			\end{alignat}
			or in the variational form:
			\eq{\label{DPz}
				\hat a(\bm v,\bm z)=(\bm v, \bm w)  \quad\forall \bm v\in \V.
			}
			Noting that $\w\cdot\boldsymbol{\nu}=0$ on $\Ga $, similar to  the proof of \eqref{eq:H2stab}, we may derive  the following $\bH^2$ regularity estimate for the above problem:
			\begin{align}
				\|{\bm z}\|_{2}&\ls \|\w\|.\label{D-regularity-H2}
			\end{align}
			From \eqref{DPz}, \eqref{approx4a}, we may deduce that
			\begin{align}
				\notag
				({\bm u}-\tilde{\bm u}_h,\w)
				&=\hat a ({\bm u}-\tilde{\bm u}_h,{\bm z}-\pi_N {\bm z})
				\lesssim   \interleave{\bm u}-\tilde{\bm u}_h\interleave \interleave {\bm z}-\pi_N {\bm z} \interleave\\
				&\ls C_{\ka h} \interleave{\bm u}-\tilde{\bm u}_h\interleave  h|\bm z|_2\ls C_{\ka h} h \interleave{\bm u}-\bm v_h\interleave \|\bm w\|. \label{Dpart4}
			\end{align}
			Since $\|\w\|^2=({\bm \Phi}_h,\w)=({\bm u}-{\bm v}_h,\w)-({\bm u}-\tilde{\bm u}_h,\w),$ we have from \eqref{Dpart4} that
			\begin{align}
				\label{w}
				\|\w\| \lesssim \|{\bm u}-{\bm v}_h\|+ C_{\ka h} h\interleave{\bm u}-{\bm v}_h\interleave,
			\end{align}
			which together with \eqref{lb} completes the proof of the lemma. 
		\end{proof}
		
		The following lemma gives an estimate of $\|{\bm u}-\tilde{{\bm u}}_h\|$. 
		
		\begin{lemma}\label{lem:l2err}
			We have
			\begin{align}
				\label{error}
				\|{\bm u}-\tilde{{\bm u}}_h\|  \lesssim C_{\ka h}\big(\|{\bm u}-{\bm v}_h\| +h\interleave{\bm u}-{\bm v}_h\interleave+h^\frac12\|({\bm u}-{\bm v}_h)_T\|_{\Ga}\big)  \quad \forall \bm v_h\in \V_h.
			\end{align}
		\end{lemma}
		\begin{proof} The idea is to convert the estimation of $\|{\bm u}-\tilde{{\bm u}}_h\| $ to that of $\|({\bm u}-\tilde{\bm u}_h)_T\|_{\Ga}$. For
			${\bm \Phi}_h=\tilde{{\bm u}}_h-{\bm v}_h$, according   to Lemma~\ref{lem:Vh0}, there exists ${\bm \Phi}_h^c\in \V_h^0$ such that
			\begin{align}
				\label{L2}
				\|{\bm \Phi}_h-{\bm \Phi}_h^c\| +h\|{\curl\,}({\bm \Phi}_h-{\bm \Phi}_h^c)\| &\lesssim h^{\frac{1}{2}}\| {\bm \Phi}_{h,T}\|_{\Ga}.
			\end{align}
			From Lemma~\ref{lem:HD} we have the following discrete Helmholtz decomposition for ${\bm \Phi}_h^c$:
			$${\bm \Phi}_h^c=\w_h^0+\nabla r_h^0,$$
			where $r_h^0\in U_h^0$ and $\w_h^0\in\V_h^0$ is discrete divergence-free. Moreover, there exists  $\w^0\in \bH_0({\curl})$ such that $\dive \w^0=0$, $\curl w^0=\curl {\bm \Phi}_h^c$, and
			\begin{align}
				\label{w1}
				\|\w_h^0-\w^0\|\lesssim h\|{\curl}\,\w_h^0\| =h\|{\curl}\,{\bm \Phi}_h^c\|.
			\end{align}
			From \eqref{eq:proj}, we know that
			\begin{align}
				\label{dis-div}
				({\bm u}-\tilde{\bm u}_h,\nabla \phi_h^0)=0  \quad \forall \phi_h^0\in U_h^0.
			\end{align}
			Next, we  introduce the following dual problem
			\begin{alignat}{2}
				\label{A-equations}
				{\curl}\,{\curl}\,{{\bm \Psi}}+\ka^2 {{\bm \Psi}}&=\w^0  \quad &&{\rm in }\ \Omega,\\
				\label{A-boundary}{\curl}\,{{\bm \Psi}}\times \boldsymbol{\nu}+{\ii}\ka \lambda {{\bm \Psi}}_T&= \bm{0} \quad  &&{\rm on }\  \Ga,
			\end{alignat}
			or in the variational form: 
			\eq{\label{DAz}
				\hat a(\bm v,\bm \Psi)=(\bm v, \bm w^0)  \quad\forall \bm v\in \V.
			}
			Similar to the proof of Lemma~\ref{thm:stability} (see also \cite[Theorem~4.3]{Melenk_2020}), we have the following  estimates:
			\begin{align}
				\label{A-regularity}
				\|{\bm \Psi}\|_{\bH^1(\curl)}&\lesssim \|\w^0\|,\\
				\label{A-regularity-H2}
				\|{\bm \Psi}\|_{2}&\ls \|\w^0\| +\ka^{-1}\|\w^0\|_{1}\lesssim \|\w^0\| +\ka^{-1}\|\curl\w^0\| \notag\\
				&=\|\w^0\| +\ka^{-1}\|{\curl}\,{\bm \Phi}_h^c\|.
			\end{align}
			Using \eqref{DAz}, we obtain
			\begin{align}
				&({\bm u}-\tilde{\bm u}_h,\w^0)
				\label{part4} =\hat a({\bm u}-\tilde{\bm u}_h,{\bm \Psi})=\hat a({\bm u}-\tilde{\bm u}_h,{\bm \Psi}-\pi_N\bm \Psi).
			\end{align}
			Using \eqref{A-regularity}--\eqref{part4} and Lemma~\ref{IE}, we conclude that
			\begin{align}\label{approx4}
				&\big|({\bm u}-\tilde{\bm u}_h,\w^0)\big|
				\lesssim \interleave{\bm u}-\tilde{\bm u}_h\interleave h |{\bm \Psi}|_{\bH^1(\curl)}+\big(\ka^2h^2\|{\bm u}-\tilde{\bm u}_h\|+\ka h^{\frac32}\|({\bm u}-\tilde{\bm u}_h)_T\|_{\Ga}\big)|{\bm \Psi}|_{2}\notag
				\\
				&\lesssim C_{\ka h}h\interleave{\bm u}-\tilde{\bm u}_h\interleave \|\w^0\|+\big( h^2\interleave{\bm u}-\tilde{\bm u}_h\interleave +  h^{\frac32}\|({\bm u}-\tilde{\bm u}_h)_T\|_{\Ga}\big)\|{\curl}\,{\bm \Phi}_h^c\|.
			\end{align}
			From \eqref{dis-div} and the orthogonality between $\w^0$ and $\nabla r_h^0$, we
			may get
			\eqn{\|{\bm u}-\tilde{\bm u}_h\|^2+\|\w^0\| ^2&=({\bm u}-\tilde{\bm u}_h+\w^0,{\bm u}-\tilde{\bm u}_h+\w^0)-2\re({\bm u}-\tilde{\bm u}_h,\w^0)\\
				&=({\bm u}-\tilde{\bm u}_h+\w^0,{\bm u}-{\bm v}_h)-({\bm u}-\tilde{\bm u}_h+\w^0,{\bm \Phi}_h-{\bm \Phi}_h^c)\\
				&\quad-({\bm u}-\tilde{\bm u}_h+\w^0,\w_h^0-\w^0)-2\re({\bm u}-\tilde{\bm u}_h,\w^0), }
			which together with  (\ref{approx4}) and the Young's inequality, gives
			\begin{align*}
				\|{\bm u}-\tilde{\bm u}_h\|^2 +\|\w^0\| ^2&\lesssim 
				\|{\bm u}-{\bm v}_h\| ^2+\|{\bm \Phi}_h-{\bm \Phi}_h^c\| ^2+\|\w_h^0-\w^0\| ^2\\
				&\quad+C_{\ka h}h^2\interleave{\bm u}-\tilde{\bm u}_h\interleave^2+  h \|({\bm u}-\tilde{\bm u}_h)_T\|_{\Ga}^2+h^2\|{\curl}\,{\bm \Phi}_h^c\| ^2.
			\end{align*}
			By \eqref{eq:Pherror_energy}, (\ref{L2}) and (\ref{w1}), we have
			\begin{align}
				\notag
				\|{\bm u}-\tilde{\bm u}_h\|+\|\w^0\| &\lesssim \|{\bm u}-{\bm v}_h\| +h^{\frac{1}{2}}\| {\bm \Phi}_{h,T} \|_{\Ga}+h\|{\curl}\,{\bm \Phi}_h^c\| \\
				\label{Hw}
				&\quad+C_{\ka h} h\interleave{\bm u}-{\bm v}_h\interleave+ h^\frac12\|({\bm u}-\tilde{\bm u}_h)_T\|_{\Ga}.
			\end{align}
			From (\ref{L2}) and \eqref{eq:Pherror_energy} we have
			\begin{align*}
				h\|{\curl}\,{\bm \Phi}_h^c\| &\leq h\|{\curl}\,({\bm \Phi}_h-{\bm \Phi}_h^c)\| +h\|{\curl}\,{\bm \Phi}_h\| \lesssim h^{\frac{1}{2}}\| {\bm \Phi}_{h,T} \|_{\Ga}+ h\interleave{\bm u}-{\bm v}_h\interleave.
			\end{align*}
			While using  the triangle inequality, we get
			\begin{align*}
				h^{\frac{1}{2}}\| {\bm \Phi}_{h,T} \|_{\Ga}&\leq h^{\frac{1}{2}}\big(\|({\bm u}-\tilde{\bm u}_h)_T\|_{\Ga}+ \|({\bm u}-{\bm v}_h)_T\|_{\Ga}\big).
			\end{align*}
			Inserting the above two inequalities into (\ref{Hw}), we obtain 
			\begin{align}
				\label{L2tilde}
				\|{\bm u}-\tilde{{\bm u}}_h\|  \lesssim&h^{\frac{1}{2}}\|({\bm u}-\tilde{\bm u}_h)_T\|_{\Ga}+h^{\frac{1}{2}}\|({\bm u}-{\bm v}_h)_T\|_{\Ga}+\|{\bm u}-{\bm v}_h\| \notag\\
				&+C_{\ka h} h\interleave{\bm u}-{\bm v}_h\interleave  \quad \forall \bm v_h\in\V_h,
			\end{align}
			which together with Lemma~\ref{lem:bd} completes the proof of the lemma.
		\end{proof}
		
		Finally, the proof of Lemma~\ref{lem:Pherror} follows by taking $\bm v_h=\pi_N{\bm u}$ in Lemmas \ref{lem:bd}--\ref{lem:l2err} and using Lemma~\ref{IE}.\hfill \qedsymbol

		\bibliographystyle{abbrv} 
		\bibliography{reference1}

\begin{thebibliography}{10}

\bibitem{adams2003sobolev}
R.~A. Adams and J.~J. Fournier.
\newblock {\em Sobolev spaces}.
\newblock Elsevier, 2003.

\bibitem{Ainsworth2004}
M.~Ainsworth.
\newblock Discrete dispersion relation for $hp$-version finite element
  approximation at high wave number.
\newblock {\em SIAM Journal on Numerical Analysis}, 42(2):553--575, 2004.

\bibitem{ainsworth2004dispersive}
M.~Ainsworth.
\newblock Dispersive properties of high--order {N}{\'e}d{\'e}lec/edge element
  approximation of the time--harmonic {M}axwell equations.
\newblock {\em Philosophical Transactions of the Royal Society of London.
  Series A: Mathematical, Physical and Engineering Sciences},
  362(1816):471--491, 2004.

\bibitem{amrouche1998vector}
C.~Amrouche, C.~Bernardi, M.~Dauge, and V.~Girault.
\newblock Vector potentials in three-dimensional non-smooth domains.
\newblock {\em Mathematical Methods in the Applied Sciences}, 21(9):823--864,
  1998.

\bibitem{Babuska1997pollution}
I.~M. Babu\v{s}ka and S.~A. Sauter.
\newblock Is the pollution effect of the {FEM} avoidable for the {H}elmholtz
  equation considering high wave numbers?
\newblock {\em SIAM Journal on Numerical Analysis}, 34(6):2392--2423, 1997.

\bibitem{brenner2008mathematical}
S.~Brenner and L.~Scott.
\newblock {\em The mathematical theory of finite element methods}.
\newblock Springer, 2008.

\bibitem{burman2005}
E.~Burman.
\newblock A unified analysis for conforming and nonconforming stabilized finite
  element methods using interior penalty.
\newblock {\em SIAM Journal on Numerical Analysis}, 43(5):2012--2033, 2005.

\bibitem{burman1D}
E.~Burman, H.~Wu, and L.~Zhu.
\newblock Linear continuous interior penalty finite element method for
  {H}elmholtz equation with high wave number: One-dimensional analysis.
\newblock {\em Numerical Methods for Partial Differential Equations}, pages
  1378--1410, 2016.

\bibitem{chen24}
Z.~Chen.
\newblock On the regularity of time-harmonic {M}axwell equations with impedance
  boundary conditions.
\newblock {\em Commun. Appl. Math. Comput.}, 2024.

\bibitem{colton1998inverse}
D.~L. Colton, R.~Kress, and R.~Kress.
\newblock {\em Inverse acoustic and electromagnetic scattering theory},
  volume~93.
\newblock Springer, 1998.

\bibitem{costabel2010corner}
M.~Costabel, M.~Dauge, and S.~Nicaise.
\newblock {Corner Singularities and Analytic Regularity for Linear Elliptic
  Systems. {P}art \uppercase \expandafter {\romannumeral 1 }: Smooth domains.}
\newblock 211 pages, Feb. 2010.

\bibitem{douglas2008interior}
J.~Douglas and T.~Dupont.
\newblock Interior penalty procedures for elliptic and parabolic {G}alerkin
  methods.
\newblock In {\em Computing Methods in Applied Sciences: Second International
  Symposium December 15--19, 1975}, pages 207--216. Springer, 2008.

\bibitem{DuWu2015}
Y.~Du and H.~Wu.
\newblock Preasymptotic error analysis of higher order {FEM} and {CIP-FEM} for
  {H}elmholtz equation with high wave number.
\newblock {\em SIAM Journal on Numerical Analysis}, 53(2):782--804, 2015.

\bibitem{feng_lu_xu_2016}
X.~Feng, P.~Lu, and X.~Xu.
\newblock A hybridizable discontinuous {G}alerkin method for the time-harmonic
  {M}axwell equations with high wave number.
\newblock {\em Computational Methods in Applied Mathematics}, 16(3):429–445,
  2016.

\bibitem{fw2011}
X.~Feng and H.~Wu.
\newblock {$hp$}-discontinuous {G}alerkin methods for the {H}elmholtz equation
  with large wave number.
\newblock {\em Math. Comp.}, 80(276):1997--2024, 2011.

\bibitem{feng2014absolutely}
X.~Feng and H.~Wu.
\newblock An absolutely stable discontinuous {G}alerkin method for the
  indefinite time-harmonic {M}axwell equations with large wave number.
\newblock {\em SIAM Journal on Numerical Analysis}, 52(5):2356--2380, 2014.

\bibitem{gatica2012finite}
G.~N. Gatica and S.~Meddahi.
\newblock Finite element analysis of a time harmonic {M}axwell problem with an
  impedance boundary condition.
\newblock {\em IMA Journal of Numerical Analysis}, 32(2):534--552, 2012.

\bibitem{Hiptmair2002femcem}
R.~Hiptmair.
\newblock Finite elements in computational electromagnetism.
\newblock {\em Acta Numerica}, 11:237–339, 2002.

\bibitem{Hiptmair2011STABILITYRF}
R.~Hiptmair, A.~Moiola, and I.~Perugia.
\newblock Stability results for the time-harmonic {M}axwell equations with
  impedance boundary conditions.
\newblock {\em Mathematical Models and Methods in Applied Sciences},
  21:2263--2287, 2011.

\bibitem{hiptmair2012}
R.~Hiptmair, A.~Moiola, and I.~Perugia.
\newblock Error analysis of {T}refftz-discontinuous {G}alerkin methods for the
  time-harmonic {M}axwell equations.
\newblock {\em Mathematics of Computation}, 82(281):247--268, 2013.

\bibitem{houston2005interior}
P.~Houston, I.~Perugia, A.~Schneebeli, and D.~Sch{\"o}tzau.
\newblock Interior penalty method for the indefinite time-harmonic maxwell
  equations.
\newblock {\em Numerische Mathematik}, 100:485--518, 2005.

\bibitem{ihlenburg98}
F.~Ihlenburg.
\newblock {\em Finite element analysis of acoustic scattering}, volume 132 of
  {\em Applied Mathematical Sciences}.
\newblock Springer-Verlag, New York, 1998.

\bibitem{jin2015theorycem}
J.-M. Jin.
\newblock {\em Theory and computation of electromagnetic fields}.
\newblock John Wiley \& Sons, 2015.

\bibitem{lu2017HDGMaxwell}
P.~Lu, H.~Chen, and W.~Qiu.
\newblock An absolutely stable $hp$-{HDG} method for the time-harmonic
  {M}axwell equations with high wave number.
\newblock {\em Mathematics of Computation}, 86(306):1553--1577, 2017.

\bibitem{lu2018regularity}
P.~Lu, Y.~Wang, and X.~Xu.
\newblock Regularity results for the time-harmonic {M}axwell equations with
  impedance boundary condition.
\newblock {\em arXiv:1804.07856v1}, 2018.

\bibitem{pplu2019}
P.~Lu, H.~Wu, and X.~Xu.
\newblock Continuous interior penalty finite element methods for the
  time-harmonic {M}axwell equation with high wave number.
\newblock {\em Adv. Comput. Math.}, 45(5–6):3265–3291, dec 2019.

\bibitem{melenk2010dtn}
J.~M. Melenk and S.~Sauter.
\newblock Convergence analysis for finite element discretizations of the
  {H}elmholtz equation with {D}irichlet-to-{N}eumann boundary conditions.
\newblock {\em Mathematics of Computation}, 79(272):1871--1914, 2010.

\bibitem{melenk2011}
J.~M. Melenk and S.~Sauter.
\newblock Wavenumber explicit convergence analysis for {G}alerkin
  discretizations of the {H}elmholtz equation.
\newblock {\em SIAM Journal on Numerical Analysis}, 49(3):1210--1243, 2011.

\bibitem{Melenk_2020}
J.~M. Melenk and S.~A. Sauter.
\newblock Wavenumber-explicit $hp$-{FEM} analysis for {M}axwell’s equations
  with transparent boundary conditions.
\newblock {\em Foundations of Computational Mathematics}, 21(1):125--241, 2020.

\bibitem{melenk2023wavenumber}
J.~M. Melenk and S.~A. Sauter.
\newblock Wavenumber-explicit $hp$-{FEM} analysis for {M}axwell’s equations
  with impedance boundary conditions.
\newblock {\em Foundations of Computational Mathematics}, 2023.

\bibitem{monk2003}
P.~Monk et~al.
\newblock {\em Finite element methods for {M}axwell's equations}.
\newblock Oxford University Press, 2003.

\bibitem{monkdispersion}
P.~B. Monk and A.~K. Parrott.
\newblock A dispersion analysis of finite element methods for {M}axwell's
  equations.
\newblock {\em SIAM Journal on Scientific Computing}, pages 916--937, 1994.

\bibitem{Naylor1999}
D.~J. Naylor.
\newblock Filling space with tetrahedra.
\newblock {\em International Journal for Numerical Methods in Engineering},
  44(10):1383--1395, 1999.

\bibitem{Nedelec1980}
J.~C. N\'{e}d\'{e}lec.
\newblock Mixed finite elements in {${\bf R}^{3}$}.
\newblock {\em Numer. Math.}, 35(3):315–341, sep 1980.

\bibitem{Nedelec1986}
J.~C. N\'{e}d\'{e}lec.
\newblock A new family of mixed finite elements in {${\bf R}^{3}$}.
\newblock {\em Numer. Math.}, 50(1):57–81, nov 1986.

\bibitem{NicaiseTomezyk2019}
S.~Nicaise and J.~Tomezyk.
\newblock {\em The time-harmonic Maxwell equations with impedance boundary
  conditions in polyhedral domains}, pages 285--340.
\newblock De Gruyter, Berlin, Boston, 2019.

\bibitem{nicaise2020convergence}
S.~Nicaise and J.~Tomezyk.
\newblock Convergence analysis of a $hp$-finite element approximation of the
  time-harmonic {M}axwell equations with impedance boundary conditions in
  domains with an analytic boundary.
\newblock {\em Numerical Methods for Partial Differential Equations},
  36(6):1868--1903, 2020.

\bibitem{pardiso}
O.~Schenk, K.~G{\"a}rtner, W.~Fichtner, and A.~D. Stricker.
\newblock {PARDISO}: a high-performance serial and parallel sparse linear
  solver in semiconductor device simulation.
\newblock {\em Future Gener. Comput. Syst.}, 18:69--78, 2001.

\bibitem{Wu2014Pre}
H.~Wu.
\newblock Pre-asymptotic error analysis of {CIP-FEM} and {FEM} for the
  {H}elmholtz equation with high wave number. {P}art {I}: linear version.
\newblock {\em IMA Journal of Numerical Analysis}, 34:1266--1288, 2014.

\bibitem{Wu2018jssx}
H.~Wu.
\newblock {FEM} and {CIP-FEM} for {H}elmholtz equation with high wave number
  (in {C}hinese).
\newblock {\em Mathematica Numerica Sinica}, 40:191--213, 2018.

\bibitem{zyphd}
Y.~Zhou.
\newblock {\em Dispersion analysis of {CIP-FEM} for {H}elmholtz problem}.
\newblock PhD thesis, Nanjing University, 2023.

\bibitem{Zhou2022Dispersion}
Y.~Zhou and H.~Wu.
\newblock Dispersion analysis of {CIP}-{FEM} for the {H}elmholtz equation.
\newblock {\em SIAM Journal on Numerical Analysis}, 61(3):1278--1292, 2023.

\bibitem{ZhuWuHDGHelmholtz}
B.~Zhu and H.~Wu.
\newblock Preasymptotic error analysis of the {HDG} method for {H}elmholtz
  equation with large wave number.
\newblock {\em Journal of scientific computing}, pages 63(1--34), 2021.

\bibitem{ZhuWu2013II}
L.~Zhu and H.~Wu.
\newblock Preasymptotic error analysis of {CIP-FEM} and {FEM} for {H}elmholtz
  equation with high wave number. {P}art {II}: $hp$ version.
\newblock {\em SIAM Journal on Numerical Analysis}, 51(3):1828--1852, 2013.

\end{thebibliography}

	\end{document}